\documentclass[a4paper]{amsart}
\usepackage[dvipsnames]{xcolor}

\usepackage[latin1]{inputenc}

\usepackage{enumerate}
\usepackage{tikz}
\usepackage{tikz-cd}
\usepackage{mathtools}
\usepackage{graphicx}
\usepackage{hyperref}
\usepackage{latexsym}
\usepackage{amssymb}
\usepackage{enumerate}
\usepackage{stix2} 
\usepackage{a4wide} 
\usepackage{csquotes}

\makeatletter
\def\@tocline#1#2#3#4#5#6#7{\relax
	\ifnum #1>\c@tocdepth 
	\else
	\par \addpenalty\@secpenalty\addvspace{#2}%
	\begingroup \hyphenpenalty\@M
	\@ifempty{#4}{%
		\@tempdima\csname r@tocindent\number#1\endcsname\relax
	}{%
		\@tempdima#4\relax
	}%
	\parindent\z@ \leftskip#3\relax \advance\leftskip\@tempdima\relax
	\rightskip\@pnumwidth plus4em \parfillskip-\@pnumwidth
	#5\leavevmode\hskip-\@tempdima
	\ifcase #1
	\or\or \hskip 1em \or \hskip 2em \else \hskip 3em \fi%
	#6\nobreak\relax
	\hfill\hbox to\@pnumwidth{\@tocpagenum{#7}}\par
	\nobreak
	\endgroup
	\fi}
\makeatother

\usetikzlibrary{fit, patterns}

\usetikzlibrary{decorations.pathmorphing}

\tikzset{%
	symbol/.style={%
		draw=none,
		every to/.append style={%
			edge node={node [sloped, allow upside down, auto=false]{$#1$}}}
	}
}

\newtheoremstyle{plain}
  {6pt}   
  {6pt}   
  {\itshape}  
  {0pt}       
  {\bfseries} 
  {.}         
  {5pt plus 1pt minus 1pt} 
  {}          

\newtheoremstyle{definition}
  {6pt}   
  {6pt}   
  {\normalfont}  
  {0pt}       
  {\bfseries} 
  {.}         
  {5pt plus 1pt minus 1pt} 
  {}          

\theoremstyle{plain}
\newtheorem*{thm*}{Theorem}
\newtheorem{thm}{Theorem}[section]
\newtheorem{prop}[thm]{Proposition}
\newtheorem{cor}[thm]{Corollary}
\newtheorem*{cor*}{Corollary}
\newtheorem{lem}[thm]{Lemma}

\theoremstyle{definition}
\newtheorem{defn}[thm]{Definition}

\newtheorem{ex}[thm]{Example}
\newtheorem{rmk}[thm]{Remark}

\numberwithin{equation}{section}

\newcommand{\emphbf}[1]{\emph{\textbf{#1}}}

\DeclareMathAlphabet{\mathpzc}{OT1}{pzc}{m}{it}

\newcommand{\rad}[1]{\radoperator(#1)}
\newcommand{\radsquare}[1]{\radoperator^2(#1)}

\def\mhyphen{{\hbox{-}}}
\newcommand{\Ima}[1]{\operatorname{Im}#1}

\newcommand{\Rest}{\operatorname{Res}}
\newcommand{\Op}{^{\operatorname{op}}}
\newcommand{\Dim}[1]{\operatorname{dim}\,\left( #1\right)}

\newcommand{\Tra}[2]{\operatorname{Tr}\left(#1,#2\right)}

\newcommand{\Modu}[1]{#1{\mhyphen\operatorname{mod}}}

\newcommand{\Homo}[3]{\operatorname{Hom}_{#1}\left( #2,#3\right)  }

\newcommand{\Exte}[4]{\operatorname{Ext}_{#1}^{#2}\left( #3,#4\right)  }

\DeclareMathOperator{\radoperator}{rad}
\DeclareMathOperator{\Kopf}{top}

\DeclareMathOperator{\Hom}{Hom}
\DeclareMathOperator{\Ext}{Ext}

\DeclareMathOperator{\End}{End}

\DeclareMathOperator{\add}{add}

\DeclareMathOperator{\proj}{proj}

\DeclareMathOperator{\op}{\normalfont{op}}

\DeclareMathOperator{\ess}{ess}

\DeclareMathOperator{\Z}{\mathbb{Z}}
\DeclareMathOperator{\Znn}{\mathbb{Z}_{\geq 0}}
\DeclareMathOperator{\N}{\mathbb{N}}
\DeclareMathOperator{\Ker}{Ker}

\begin{document}

\title{Exact Borel subalgebras, idempotent quotients and idempotent subalgebras}

\author{Teresa Conde}

\address{Teresa Conde, Faculty of Mathematics, University of Bielefeld, Universit\"atsstra{\ss}e 25, 33615 Bielefeld, Germany}
\email{tconde@math.uni-bielefeld.de}

\author{Julian K\"ulshammer}
\address{Julian K\"ulshammer, Department of Mathematics, Uppsala University, Box 480, 751 06 Uppsala, Sweden}
\email{julian.kuelshammer@math.uu.se}

\date{\today}

\thanks{Teresa Conde acknowledges support by the Deutsche Forschungsgemeinschaft (DFG, German Research Foundation) -- Project-ID 491392403 - TRR 358. Julian K\"ulshammer acknowledges support by the Swedish Research Council -- registration number 2024-05357. He would also like to thank Henning Krause for posing a question at ICRA 2018 in Prague that lead to the present article. Furthermore, he thanks Walter Mazorchuk for answering some questions about BGG category $\mathcal{O}$. In addition, he would like to thank Anna Rodriguez Rasmussen for providing the computation for the regular exact Borel subalgebra of the principal block of $\mathbb{A}_1\times \mathbb{A}_1$. Lastly, both authors would like to thank the referee for a careful read and several suggestions on how to improve the exposition of the paper.}

\begin{abstract}
This article studies the compatibility of Koenig's notion of an exact Borel subalgebra of a quasi-hereditary or, more generally, standardly stratified algebra with taking idempotent subalgebras or quotients. As an application, we provide bounds for  the multiplicities of indecomposable projectives in the principal blocks of BGG category $\mathcal{O}$ having basic regular exact Borel subalgebras.  
\end{abstract}

\maketitle

\tableofcontents

\section{Introduction}

The term `idempotent' was coined by Peirce in \cite{Pei81} in his work on classifying algebras of dimension smaller than $6$. Already in his work, the importance of idempotents for the study of decompositions was realised. By now, many variants of such decompositions are known and well-studied: central idempotents in rings give rise to block decompositions; idempotents in an algebra give rise to a decomposition into its projective direct summands; idempotents in an endomorphism ring correspond bijectively to decompositions of a module; an idempotent in an algebra yields a recollement of module categories---an algebraic analogue of Grothendieck's $6$-functor formalism; algebras with sufficiently many idempotents provide an alternative way to speak about $\Bbbk$-linear categories. 

In addition, many classes of algebras come with a (natural) set of idempotents, for example (admissible quotients of) path algebras of quivers, Graham and Lehrer's cellular algebras, and  Beilinson--Lusztig--MacPherson's idempotented quantum groups. The algebras of interest to us are quasi-he\-re\-di\-ta\-ry algebras, introduced in \cite{Sco87, CPS88}, and their generalisations to standardly stratified algebras (see \cite{CPS96} and independently \cite{D96}). Examples of these classes of algebras include blocks of BGG category $\mathcal{O}$ and generalisations (see \cite{FKM01}), Schur algebras, as well as algebras of global dimension smaller than or equal to two. There are several alternatives to define these notions, one of which is ring-theoretical using a chain of idempotent ideals, called heredity chain in the case of quasi-hereditary algebras and a standard stratification in the case of standardly stratified algebras. 

Given an idempotent $e$ in a standardly stratified algebra $A$, in light of the inductive definition in terms of idempotent ideals, it is natural to ask for conditions under which the idempotent subalgebra $eAe$ and the idempotent quotient $A/AeA$ are standardly stratified. Several authors have investigated this topic yielding different criteria starting with Dlab and Ringel in \cite[Theorem 1]{DR89} in the case of quasi-hereditary algebras. More recently, slightly different criteria have been investigated by Gao, Koenig and Psaroudakis in \cite[Theorem 2.1]{GKP19}. Our own variant, generalised to standardly stratified algebras, is as follows.

\begin{thm*}
Let $A$ be a left standardly stratified algebra and let $e\in A$ be an idempotent. The following conditions are equivalent:
\begin{itemize}
\item There exists a left standard stratification of $A$ in which the ideal $AeA$ appears.
\item The support of $e$ is a coideal of the partially ordered set of isomorphism classes of simples, equipped with the essential order (i.e.~the coarsest partial order on the simples yielding the same standard modules).
\item The $A$-module $A/AeA$ has a $\Delta$-filtration while the $A$-module $D(A/AeA)$ has a $\bar{\nabla}$-filtration.
\end{itemize}
In this case, the algebras $A/AeA$ and $eAe$ are left standardly stratified algebras. 
\end{thm*}

Exact Borel subalgebras were introduced by Koenig in \cite{Koe95} in an attempt to abstract the properties of the universal enveloping algebra of a Borel subalgebra of a semisimple Lie algebra $\mathfrak{g}$, as a subalgebra of the universal enveloping algebra of $\mathfrak{g}$. Generalisations to standardly stratified algebras have been proposed in \cite{KM02}, \cite{CZ19}, \cite{BPS23}, and \cite{Got24}. We follow the definitions in the latter two papers, but use a slight generalisation of the notion to non-algebraically closed fields, adapted from \cite{Conde2021}. 
In \cite{KKO14}, together with Koenig and Ovsienko, the second author proved the existence of exact Borel subalgebras for quasi-hereditary algebras up to Morita equivalence. For standardly stratified algebras, variants have been provided in \cite{BPS23} and \cite{Got24}. Furthermore, work of Rodriguez Rasmussen \cite{RR25}, improving on  joint work \cite{KM23} of the second author with Miemietz (see also \cite{BKK20} and \cite{Conde2021b}), proved  uniqueness of exact Borel subalgebras up to conjugation under the additional assumption of (Kleiner--Roiter) regularity, see \cite{KR77} or Definition \ref{def:homological-regular}. Here, we therefore investigate how (regular) exact Borel subalgebras behave under taking idempotent subalgebras and idempotent quotients. 

\begin{thm*}
Let $A$ be a left standardly stratified algebra with an exact Borel subalgebra $B$. Let $e\in B$  be an idempotent supported in a coideal of the partially ordered set of isomorphism classes of simple modules. Then $eBe$ is an exact Borel subalgebra of $eAe$ and $B/BeB$ is an exact Borel subalgebra of $A/AeA$. If $B$ in $A$ is regular exact, then so are $eBe$ and $B/BeB$.  
\end{thm*}

In \cite{Conde2021}, the first author provided an inductive formula for the decomposition multiplicities of the Morita representative of a quasi-hereditary algebra $A$ containing a basic regular exact Borel subalgebra. Her approach uses a particular lower triangular matrix $V_A$. Using the preceding theorem, we prove that given an idempotent $e$ as above, the matrix $V_A$ is block lower triangular with block diagonal entries given by $V_{A/AeA}$ and $V_{eAe}$. 

Already in his first paper on exact Borel subalgebras \cite{Koe95}, Koenig proved that for each block $\mathcal{O}_\chi$ of BGG category $\mathcal{O}$, there is an algebra $A_\chi$ with an exact Borel subalgebra such that $\Modu{A_\chi}\simeq \mathcal{O}_\chi$. The notion of regularity was not considered at the time. It is also open whether $A_\chi$ can be chosen to be basic. As an application of our results, we obtain that the Morita representatives of $A_\chi$ admitting regular exact Borel subalgebras have high decomposition multiplicities. 

\begin{cor*}
Let $\mathfrak{g}\neq \mathfrak{sl}_2$ be a finite-dimensional simple complex Lie algebra. Let $A$ be the algebra such that $\Modu{A}$ is equivalent to the principal block of BGG category $\mathcal{O}$ and $A$ has a basic regular exact Borel subalgebra. Then there exists a projective indecomposable module whose multiplicity in $A$ is at least $9$. 
\end{cor*}

The structure of the paper is as follows. In Section \ref{sec:standardly-stratified} we recall the definition of a standardly stratified algebra and the associated essential order. Section \ref{sec:compatible-idempotents} is devoted to discussing the different possibilities of defining compatibility of an idempotent with the partial order of a standardly stratified algebra. This discussion is accompanied by several examples and culminates in the proof of Theorem \ref{thm:defofcompatibility}, from which our first main theorem follows. In Section \ref{sec:exact-borels} we recall the basic theory of exact Borel subalgebras and discuss certain shortcomings of the definition. In Section \ref{sec:exact-borels-compatible} we discuss compatibility of exact Borel subalgebras with taking idempotent subalgebras and idempotent quotients. In particular, we prove our second main theorem, which is derived from Theorem \ref{thm:last_firstpart} (see also Propositions \ref{prop:compatibility-subalgebra} and \ref{prop:compatibility-quotient}). In the final section, Section \ref{sec:decomposition-multiplicities}, we provide some results regarding the decomposition multiplicities of projectives in the quasi-hereditary algebras containing basic regular exact Borel subalgebras (Propositions \ref{prop:compatibility-matrix} and \ref{prop:last}). These results are then applied to the case of BGG category $\mathcal{O}$ to prove the aforementioned corollary (Corollary \ref{cor:last}). 

\subsection*{Notation and conventions}

Let $\Bbbk$ denote a fixed field. Unless explicitly stated otherwise, the word `algebra' will mean a finite-dimensional $\Bbbk$-algebra and all modules are assumed to be finite-dimensional left modules. Given an algebra $A$, we shall denote the category of (finite-dimensional left) $A$-modules by $\Modu{A}$.

The isomorphism classes of the simple $A$-modules may be labelled by the elements of a finite set $Q_0$. 
We use typewriter font to denote the elements $\mathtt{i}$, $\mathtt{j}$, $\mathtt{k}$,\,\dots\, of $Q_0$. For $\mathtt{i}\in Q_0$ denote the simple $A$-modules by $L_\mathtt{i}$ or $L_\mathtt{i}^A$, and use the notation $P_\mathtt{i}$ or $P_\mathtt{i}^A$ (resp.~$I_\mathtt{i}$ or $I_\mathtt{i}^A$) for the projective cover (resp.~injective hull) of $L_\mathtt{i}$. Write $\{e_\mathtt{i}\mid \mathtt{i} \in Q_0\}$ for a set of (primitive) idempotents in $A$ satisfying $A e_\mathtt{i} \cong P_\mathtt{i}$ for every $\mathtt{i}\in Q_0$. Finally, denote the multiplicity of the simple $L_\mathtt{i}$ as a composition factor of a module $X$ by $[X:L_\mathtt{i}]$. The set $Q_0$ will usually be endowed with a partial order $\unlhd$. Given $\mathtt{i},\mathtt{j}\in Q_0$, we write $\mathtt{i} \lhd \mathtt{j}$ if $\mathtt{i}\unlhd \mathtt{j}$ and $\mathtt{i}\neq \mathtt{j}$. 

The arrows in a quiver shall be composed from right to left. For a vector space $M$, we write $D(M)$ to denote the $\Bbbk$-dual of $M$.

\section{Standardly stratified algebras}
\label{sec:standardly-stratified}

Standardly stratified algebras were introduced by Cline, Parshall, and Scott in \cite{CPS96} (or rather a version based on a preorder instead of a partial order) generalising their notion of quasi-hereditary algebra in \cite{Sco87, CPS88}. Independently, they were also introduced (under the more restrictive condition of a total order) in \cite{D96} under the name $\Delta$-filtered algebras\footnote{The example in \cite[Section 8.2]{FRISK2007} shows that the class of standardly stratified algebras based on a preorder considered in \cite{CPS96} is, `structure-wise', strictly larger than the one in \cite{D96,ADL98, AHLU00}. This general version of standardly stratified algebras is crucially used in \cite{WEBB20084073}.}. Their basic properties were later explored in a series of papers, see e.g.~\cite{ADL98, AHLU00}. In order to define standardly stratified algebras, one needs the notions of standard and costandard module, as well as the corresponding proper versions.

\begin{defn}
	Let $(Q_0,\unlhd)$ be a poset labelling the isomorphism classes of the simple modules over an algebra $A$.
	\begin{enumerate}
		\item The \emphbf{standard module} $\Delta_\mathtt{i}$, $\mathtt{i} \in Q_0$, is the largest quotient of $P_\mathtt{i}$ whose composition factors are all of the form $L_\mathtt{j}$ with $\mathtt{j}\unlhd \mathtt{i}$. 
		\item The \emphbf{proper standard module} $\bar{\Delta}_\mathtt{i}$, $\mathtt{i} \in Q_0$, is the largest quotient $X$ of $\Delta_\mathtt{i}$ with $[X:L_\mathtt{i}]=1$.
		\item The \emphbf{costandard module} $\nabla_\mathtt{i}$, $\mathtt{i} \in Q_0$, is the largest submodule of $I_\mathtt{i}$ whose composition factors are all of the form $L_\mathtt{j}$ with $\mathtt{j}\unlhd \mathtt{i}$. 
		\item The \emphbf{proper costandard module} $\bar{\nabla}_\mathtt{i}$, $\mathtt{i} \in Q_0$, is the largest submodule $X$ of the costandard module $\nabla_\mathtt{i}$ with $[X:L_\mathtt{i}]=1$.
	\end{enumerate}
\end{defn}

The standard and costandard modules, as well as the corresponding proper versions, shall be decorated with a superscript to indicate the ambient algebra whenever this is unclear from the context.

Given an algebra $A$ endowed with a poset $(Q_0, \unlhd)$ labelling the isomorphism classes of simple $A$-modules, denote by $\Delta$ the set of all standard modules and define $\bar{\Delta}$, $\nabla$ and $\bar{\nabla}$ in a similar way. A chain of $A$-submodules whose subquotients are isomorphic to modules in some set $\Theta$ is called a \emphbf{$\Theta$-filtration}, and the (exact) full subcategory of all modules having a $\Theta$-filtration shall be denoted by $\mathcal{F}(\Theta)$. 

The labelling poset $(Q_0, \unlhd)$ is said to be \emphbf{adapted} to $A$ provided that the following holds: for every $A$-module $X$ with simple top $L_\mathtt{i}$ and simple socle $L_\mathtt{j}$ where $\mathtt{i}$ and $\mathtt{j}$ are incomparable elements of $(Q_0, \unlhd)$ there exists $\mathtt{k} \in Q_0$ with $[X: L_\mathtt{k}] \neq 0$ and with $\mathtt{k}\rhd \mathtt{i}$ or $\mathtt{k}\rhd \mathtt{j}$ (equivalently there exists $\mathtt{k}$ with $[X:L_\mathtt{k}]\neq 0$ and $\mathtt{k}\rhd \mathtt{i}$ and $\mathtt{k}\rhd \mathtt{j}$). It follows from \cite[Section 1]{DR92} that any refinement of a labelling poset adapted to $A$ is still adapted to $A$, and the standard and costandard modules with respect to the refinement (as well as their proper versions) coincide with those corresponding to the original poset. For a discussion of equivalent definitions of adapted orders, see \cite{RR23}.

We recall the notions of a standardly stratified algebra and a quasi-hereditary algebra. For introductions into the theory of quasi-hereditary algebras, the reader can consult the survey articles \cite{DR92, KK99}. For standardly stratified algebras, we suggest to start by reading \cite{AHLU00}.   

\begin{defn}
\label{defn:stdd_stratified_algebra}
	Let $(Q_0,\unlhd)$ be a poset labelling the isomorphism classes of the simple modules over an algebra $A$.
	\begin{enumerate}
		\item The algebra $A$ is \emphbf{left standardly stratified} with respect to $(Q_0, \unlhd)$ if, for every $\mathtt{i}\in Q_0$, the kernel of $P_\mathtt{i}\twoheadrightarrow \Delta_\mathtt{i}$ has a filtration whose subquotients are of the form $\Delta_\mathtt{j}$, with $\mathtt{j}\rhd \mathtt{i}$.
		\item The algebra $A$ is \emphbf{right standardly stratified} with respect to $(Q_0, \unlhd)$ if $A\Op$ is left standardly stratified with respect to $(Q_0, \unlhd)$; in other words, if the cokernel of $\nabla_\mathtt{i}\hookrightarrow I_\mathtt{i}$ has a filtration whose subquotients are of the form $\nabla_\mathtt{j}$, with $\mathtt{j}\rhd \mathtt{i}$.
		\item The algebra $A$ is \emphbf{quasi-hereditary} with respect to $(Q_0,\unlhd)$ if it is right standardly stratified and if $\Delta_\mathtt{i}=\bar{\Delta}_\mathtt{i}$ for every $\mathtt{i}\in Q_0$. Equivalently said, quasi-hereditary algebras are the left standardly stratified algebras with $\nabla_\mathtt{i}=\bar{\nabla}_\mathtt{i}$ for every $\mathtt{i} \in Q_0$.
	\end{enumerate} 
\end{defn}

Let us briefly mention that there are several equivalent conditions characterising left and right standardly stratified (resp.~quasi-hereditary) algebras. For example, in \cite[Lemma 2.2]{D96} it is shown that $A$ is left standardly stratified if and only if the indecomposable injective $A$-modules $I_\mathtt{i}$ have an analogous filtration by proper costandard modules. In \cite[Theorem 2.2.3]{CPS96} a characterisation in terms of standard stratifications is provided -- we  will touch on 
 this perspective later on when analysing compatibility of idempotents with standardly stratified structures, see Definition \ref{defn:compatibility}. 

We shall say that an algebra is \emphbf{standardly stratified} if it is either left or right standardly stratified. We will mostly focus on left standardly stratified algebras. Most of the notions introduced for left standardly stratified algebras carry over to right standardly stratified algebras by simply replacing standard modules by proper standard modules, and proper costandard modules by costandard modules. In fact, it is possible to modify many definitions and results in this paper so that the mixed standardly stratified case is accommodated (see \cite[Definition 1.3]{ADL98} and also \cite{BS24,AT24}), but details will not be provided.

\begin{defn}
\label{defn:equivalent}
Two left standardly stratified algebras $(A,Q_0, \unlhd)$ and $(A',Q_0', \unlhd')$ are \emphbf{equivalent} if there is an equivalence of exact categories between the corresponding categories $\mathcal{F}(\Delta^A)$ and $\mathcal{F}(\Delta^{A'})$ of modules filtered by standard modules. Similarly, two right standardly stratified algebras are equivalent if the respective exact categories of modules filtered by proper standard modules are equivalent.
\end{defn}

It follows from the Standardisation Theorem (see \cite[Theorem 2]{DR92} and \cite{ErdmannSaenz03,AT24} for generalisations) that any two equivalent standardly stratified algebras are Morita equivalent and have essentially the same standardly stratified structure. The next example shows why it is crucial that the equivalence of categories in Definition \ref{defn:equivalent} transports the exact structure. 

\begin{ex}
    Consider $A=\Bbbk Q$, where 
    \[Q=
    \begin{tikzcd}[ column sep = small]
        \overset{\mathtt{1}}{\circ} \ar[r] & \overset{\mathtt{2}}{\circ}
    \end{tikzcd}.\]
    Endow $Q_0$ with the usual linear order on $\{\mathtt{1},\mathtt{2}\}$, so that $A$ becomes quasi-hereditary with simple standard modules. Note that $A$ has finite type and its Auslander algebra is given by $A'=\Bbbk Q'/\radsquare{\Bbbk Q'}$, where \[Q'=
    \begin{tikzcd}[ column sep = small]
        \overset{\mathtt{3}}{\circ} \ar[r] & \overset{\mathtt{2}}{\circ} \ar[r] & \overset{\mathtt{1}}{\circ}
    \end{tikzcd}.
    \]
    Endow the simple $A'$-modules with the canonical order on $\{\mathtt{1},\mathtt{2},\mathtt{3}\}$. In this way, $A'$ becomes quasi-hereditary with $\mathcal{F}(\Delta^{A'})=\proj A'$. Note that $\proj A'$ is equivalent to $\Modu{A}=\mathcal{F}(\Delta^A)$. However $\mathcal{F}(\Delta^A)$ and $\mathcal{F}(\Delta^{A'})$ cannot be equivalent as exact categories since the number of isomorphism classes of indecomposable projective objects in each of the respective exact structures does not match (one can also argue with the number of simple objects with respect to each of the exact structures). 
    
    More generally, let $A$ be a representation-directed algebra (see \cite[Section IX.3]{ASS06} for an introduction to the theory), that is a representation-finite algebra whose Auslander algebra $A'$ has an acyclic Gabriel quiver (e.g. a Dynkin quiver with arbitrary orientation). As the Gabriel quiver of $A'$ is acyclic, $A'$ admits a quasi-hereditary structure whose corresponding standard modules are projective, so that $\mathcal{F}(\Delta^{A'})=\proj A'$. On the other hand, as a representation-directed algebra, even the Gabriel quiver of $A$ is acyclic. This means that $A$ admits a quasi-hereditary structure with simple standard modules, so that $\mathcal{F}(\Delta^A)=\Modu{A}$. Now, just as categories, $\proj A'$ and $\Modu{A}$ are equivalent. However, their exact structures only coincide when $A$ is semisimple. (One has $A$ as a projective generator, while the other one has $A'$ as a projective generator.)
\end{ex}

It is not hard to check that the underlying poset of a standardly stratifyied algebra is adapted to the algebra. To see this, note that the argument in the proof of \cite[Proposition 1.4.12]{CondePhD} works for left standardly stratified algebras (see also \cite[Lemma 2.12]{FKR22}) and observe that the notion of adapted order is left-right symmetric. By the previous discussion about adapted orders, one can always refine the underlying poset of a standardly stratified algebra and even assume that it is a total order without changing the standard or the costandard modules (see also \cite[Lemma 8]{Fri04}). We shall not do so as in practice it is often more convenient to work with an honest partial order. Taking this to the extreme, there is a unique minimal partial order with the same standard and proper costandard modules, sometimes called the essential order as defined in \cite[Definition 1.2.5]{Cou20} for quasi-hereditary algebras and then generalised in \cite[Section 2.2]{MP22} to the case of left standardly stratified algebras. 

\begin{defn}
Given a left standardly stratified algebra $(A,Q_0, \unlhd)$, the \emphbf{essential order} $\unlhd_{\ess}$ on $Q_0$ is the minimal partial order such that $\mathtt{j}\unlhd_{\ess} \mathtt{i}$ if  $[\Delta_\mathtt{i}:L_\mathtt{j}]\neq 0$ or $[\bar{\nabla}_\mathtt{i}:L_\mathtt{j}]\neq 0$. The essential order of a right standardly stratified algebra is defined dually by means of proper standard modules and costandard modules.
\end{defn}

To sum up, replacing the underlying poset $(Q_0,\unlhd)$ of a standardly stratified algebra by the corresponding essential order $(Q_0,\unlhd_{\ess})$ or by any refinement of $(Q_0,\unlhd_{\ess})$ does not alter the standardly stratified structure and produces equivalent standardly stratified algebras where the associated equivalence of categories is the identity functor.

\section{Compatible idempotents}
\label{sec:compatible-idempotents}
Let $A$ be an algebra whose isomorphism classes of simples are indexed by the elements of a set $Q_0$. The \emphbf{support} of an idempotent $e\in A$ is defined as
\[\{\mathtt{i}\in Q_0\mid eL_\mathtt{i}\neq 0\}.\]
If $e$ has support $Q_0'\subseteq Q_0$, we say that $e$ is \emphbf{supported} in $Q_0'$. Note that an idempotent $e\in A$ supported in $Q_0'\subseteq Q_0$ does not necessarily satisfy $(1-e)L_\mathtt{i}=0$ for every $\mathtt{i} \in Q_0'$ as the following example shows.

\begin{ex} 
Let $A=M_2(\Bbbk)$ be the algebra of $2\times 2$-matrices. Then there is exactly one simple $A$-module $L$ up to isomorphism. Let $E_{11}$ be the idempotent of $A$ given by the elementary matrix with a $1$ in the upper left corner. Then clearly $E_{11}L\neq 0$, but also $(1-E_{11})L=E_{22}L\neq 0$. 
\end{ex}

The statement does however hold for basic algebras. If $A$ is basic, then $A\cong \bigoplus_{\mathtt{i}\in Q_0} Ae_\mathtt{i}$ with the $Ae_\mathtt{i}$ pairwise non-isomorphic. In this case, $Ae\cong \bigoplus_{\mathtt{i}\in Q_0'} Ae_\mathtt{i}$ where $Q_0'$ is the support of $e$. It follows that $A(1-e)\cong \bigoplus_{\mathtt{j}\in Q_0\setminus Q_0'} Ae_\mathtt{j}$, whence $(1-e)L_\mathtt{i}\cong\Hom_A(A(1-e),L_\mathtt{i})=0$ for all $\mathtt{i}\in Q_0'$.

We recall the definiton of recollement as defined in the context of triangulated categories in \cite{BBD82}. For a survey on the theory of recollements of abelian categories, we refer the reader to \cite{Psa18}.
\begin{defn}
	A \emphbf{recollement} of abelian categories $\mathcal{A}$, $\mathcal{B}$ and $\mathcal{C}$ is a diagram of functors
	\begin{equation}
	\label{eq:recol}
	\begin{tikzcd}[row sep=6.5 em, column sep=6.0 em]
		\mathcal{A} \arrow{r}[description]{i} & \mathcal{B} \arrow{r}[description]{e}\arrow[l,"q"', bend right] \arrow[l, "p",bend left]& \mathcal{C} \arrow[l ,"l"',bend right] \arrow[l, "r",bend left]
	\end{tikzcd}
	\end{equation}
	 such that the following conditions are satisfied:
	\begin{enumerate}
		\item $(q,i,p)$ and $(l,e,r)$ are adjoint triples;
		\item the functors $i$, $l$ and $r$ are fully faithful;
		\item the essential image $\Ima{i}$ of $i$ coincides with the kernel $\Ker{e}$ of $e$.
	\end{enumerate}
\end{defn}
Since the functors $i$ and $e$ in the definition above are simultaneously left and right adjoints, they preserve all limits and colimits and are, in particular, exact. Hence, a recollement is a diagram of functors as in \eqref{eq:recol} where the top part is a colocalisation sequence and the bottom part is a localisation sequence -- we refer to \cite[Definitions 2.1, 2.2]{Kra17} for the notion of localisation and colocatisation sequence. If one restricts to the case where the category in the centre of the diagram is the module category of a finite-dimensional algebra, these are essentially given by an idempotent as follows (see \cite[Corollary 5.5]{PV14} and also \cite[Lemma 2.5]{Kra17}): given an algebra $A$ and an idempotent $e \in A$, we have an associated recollement
    \begin{equation}
	\label{eq:recol-modules}
    \begin{tikzcd}[row sep=6.5 em, column sep=6.0 em]
    \Modu{A/AeA} \arrow{r}[description]{i} & \Modu{A} \arrow{r}[description]{e (-)}\arrow[l,"(A/Ae A)\otimes_A -"', bend right] \arrow[l, "\Homo{A}{A/Ae A}{-}",bend left]& \Modu{eAe} \arrow[l ,"Ae\otimes_{e Ae}-"',bend right] \arrow[l, "\Homo{e A e}{e A}{-}",bend left]
    \end{tikzcd}.
    \end{equation}

In this case, the set $Q_0$ gets partitioned into $Q_0'$, the support of $e$, and $Q_0\setminus Q_0'$. For the corresponding simple modules, by definition, $eL_\mathtt{i}\neq 0$ if and only if $\mathtt{i}\in Q_0'$, for which $eL_\mathtt{i}\cong L_\mathtt{i}^{eAe}$, i.e.~the set $Q_0'$ provides an indexing set for the isomorphism classes of simple $eAe$-modules. The corresponding projective covers and injective envelopes are related, by $P_\mathtt{i}=Ae\otimes_{e Ae}P_\mathtt{i}^{eAe}$ and $I_\mathtt{i}=\Hom_{eAe}(eA, I_\mathtt{i}^{eAe})$, respectively, for $\mathtt{i}\in Q_0'$. On the other hand, $eL_\mathtt{i}=0$ if and only if $L_\mathtt{i}\cong i(L_\mathtt{i}^{A/AeA})$, i.e.~the set $Q_0\setminus Q_0'$ indexes the simple $A/AeA$-modules, cf.~\cite[Proposition 2.9]{Psa18}, and the corresponding projective covers and injective envelopes are related by $P_\mathtt{i}^{A/AeA}=(A/AeA)\otimes_A P_{\mathtt{i}}$ and $I_\mathtt{i}^{A/AeA}=\Homo{A}{A/Ae A}{I_{\mathtt{i}}}$, respectively. 

Later on, we shall assume that $(A, Q_0, \unlhd)$ is standardly stratified and we ask what it means for the idempotent $e$ to be compatible with the standardly stratified structure of $A$. The methods used work very abstractly. We have tried to highlight the generality of the arguments in the proof of Lemma \ref{lem:comp-first}. Essentially, we only need properties of localisation and colocalisation sequences of `good-enough' abelian categories, together with properties of the stratified category in the middle. Here, for partial results, a `good-enough' abelian category could be a locally finite Grothendieck category or its dual version (we refer to \cite{Herzog1997} for details).

\begin{lem}
\label{lem:comp-first}
    Let $A$ be an algebra whose isomorphism classes of simple modules are labelled by the elements of a poset $(Q_0, \unlhd)$ and let $e\in A$ be an idempotent supported in some subset $Q_0'$ of $Q_0$. For every $\mathtt{i}\in Q_0\setminus Q_0'$, the following identities hold:
    \begin{enumerate}
        \item $\Delta_\mathtt{i}^{A/AeA}=(A/Ae A)\otimes_A\Delta_\mathtt{i}$ and $\bar{\Delta}_\mathtt{i}^{A/AeA}=(A/Ae A)\otimes_A\bar{\Delta}_\mathtt{i}$;
        \item $\nabla_\mathtt{i}^{A/AeA}=\Hom_A(A/AeA,\nabla_\mathtt{i})$ and $\bar{\nabla}_\mathtt{i}^{A/AeA}=\Hom_A(A/AeA,\bar{\nabla}_\mathtt{i})$.
    \end{enumerate}
    For $\mathtt{i}\in Q_0'$ and $X$ in $\Modu{A}$, $[X:L_\mathtt{i}]=[eX:L_\mathtt{i}^{eAe}]$. Moreover:
    \begin{enumerate}
        \item[(3)] if $X$ has top $L_\mathtt{i}$, $e X$ has top $L_\mathtt{i}^{eAe}$;
        \item[(4)] if $X$ has socle $L_\mathtt{i}$, $eX$ has socle $L_\mathtt{i}^{eAe}$.
    \end{enumerate}
\end{lem}

\begin{proof}
    Let $X$ be an arbitrary $A$-module. There is an epimorphism $X\twoheadrightarrow i((A/Ae A)\otimes_A X)$ (see e.g. \cite[Proposition 2.6]{Psa14}). Hence, there is an epimorphism from $\Hom_A(P_\mathtt{j},X)$ to 
    \[\Hom_A(P_\mathtt{j},i((A/Ae A)\otimes_A X))\cong\Hom_{A/AeA}((A/Ae A)\otimes_A P_\mathtt{j},(A/Ae A)\otimes_A X).\]
    Recall that $P_\mathtt{j}^{A/AeA}=(A/Ae A)\otimes_A P_\mathtt{j}$ for $\mathtt{j}\in Q_0\setminus Q_0'$, and these form a complete set of projective indecomposable $A/AeA$-modules. Moreover, $L_\mathtt{j}=i( L_\mathtt{j}^{A/AeA})$ for $\mathtt{j}\in Q_0\setminus Q_0'$ so \[\Hom_A(P_\mathtt{j}, L_\mathtt{j})= \Hom_A(P_\mathtt{j}, i(L_\mathtt{j}^{A/AeA}))\cong \Hom_{A/AeA}(P_\mathtt{j}^{A/AeA},L_\mathtt{j}^{A/AeA})\] and $(A/Ae A)\otimes_A L_\mathtt{j}= L_\mathtt{j}^{A/AeA}$. Hence, the multiplicity of $L_\mathtt{j}^{A/AeA}$ in $(A/Ae A)\otimes_A X$ is less than or equal to the multiplicity of $L_\mathtt{j}$ in $X$ for every $\mathtt{j}\in Q_0\setminus Q_0'$ and $(A/Ae A)\otimes_A X\neq 0$ if there exists an epimorphism from $X$ to some $L_\mathtt{i}$ with $\mathtt{i}\in Q_0\setminus Q_0'$. Fix now $\mathtt{i}\in Q_0\setminus Q_0'$ and set $X=\Delta_\mathtt{i}$. The functor $(A/AeA)\otimes_A -$ maps the projective presentation
    \[
    \begin{tikzcd}
        \bigoplus_{\substack{\mathtt{j} \in Q_0 \\ \mathtt{j} \not\unlhd \mathtt{i}}} P_\mathtt{j}^{\oplus a_\mathtt{j}} \ar[r] & P_\mathtt{i} \ar[r, "\pi_\mathtt{i}"] & \Delta_\mathtt{i} \ar[r] & 0
    \end{tikzcd}
    \]
    to the projective presentation
    \[
    \begin{tikzcd}
        \bigoplus_{\substack{\mathtt{j} \in Q_0\setminus Q_0' \\ \mathtt{j} \not\unlhd \mathtt{i}}} (P_\mathtt{j}^{A/AeA})^{\oplus a_\mathtt{j}} \ar[r] &  P_\mathtt{i}^{A/AeA} \ar[r] & (A/AeA)\otimes_A \Delta_\mathtt{i} \ar[r] & 0
    \end{tikzcd}.
    \]
    By the previous considerations, all composition factors of $(A/AeA)\otimes_A \Delta_\mathtt{i}$ are of the form $L_\mathtt{j}^{A/AeA}$ with $\mathtt{j} \in Q_0\setminus Q_0'$ and $\mathtt{j}\unlhd \mathtt{i}$. So $(A/AeA)\otimes_A \Delta_\mathtt{i}$ must be a quotient of $\Delta_\mathtt{i}^{A/AeA}$. However, the kernel of $(A/AeA)\otimes_A \pi_\mathtt{i}$ is generated by projectives labelled by $\mathtt{j} \in Q_0\setminus Q_0'$ with $\mathtt{j}\ntrianglelefteq{\mathtt{i}}$, so the module $(A/AeA)\otimes_A \Delta_\mathtt{i}$ cannot be a proper quotient of $\Delta_\mathtt{i}^{A/AeA}$. Thus $\Delta_\mathtt{i}^{A/AeA}=(A/AeA)\otimes_A \Delta_\mathtt{i}$ for $\mathtt{i} \in Q_0\setminus Q_0'$. To show that $\bar{\Delta}_\mathtt{i}^{A/AeA}=(A/Ae A)\otimes_A\bar{\Delta}_\mathtt{i}$, one can use a similar argument, now noting that the leftmost term of the projective presentation of $\bar{\Delta}_\mathtt{i}$ is a direct sum of projectives of the form $P_\mathtt{j}$ with $\mathtt{j} \ntriangleleft \mathtt{i}$ and that  $(A/Ae A)\otimes_A \bar{\Delta}_{\mathtt{i}}\neq 0$. To show that $\nabla_\mathtt{i}^{A/AeA}=\Hom_A(A/AeA,\nabla_\mathtt{i})$ and $\bar{\nabla}_\mathtt{i}^{A/AeA}=\Hom_A(A/AeA,\bar{\nabla}_\mathtt{i})$ one can adopt analogous (dual) strategies using the localisation part of the recollement (i.e.~the bottom half). This shows (1) and (2).
    
   To prove the remaining assertions, fix some $A$-module $X$. Recall that $I_\mathtt{j}= \Hom_{eAe}(eA,I_\mathtt{j}^{eAe})$ for every $\mathtt{j}\in Q_0'$ and note that
    \[
        \Hom_{eAe}(eX, I_\mathtt{j}^{eAe})\cong \Hom_{A}(X, \Hom_{eAe}(eA,I_\mathtt{j}^{eAe})) = \Hom_{A}(X, I_\mathtt{j}).
    \]
    Furthermore, $L_\mathtt{j}^{eAe}=e L_\mathtt{j}$ for $\mathtt{j}\in Q_0'$. So $\Hom_{eAe}(L_\mathtt{j}^{eAe}, I_\mathtt{j}^{eAe})\cong\Hom_{A}(L_\mathtt{j}, I_\mathtt{j})$, $\mathtt{j}\in Q_0'$. It follows that the multiplicity of $L_\mathtt{j}^{eAe}$ as a composition factor of $eX$ coincides with the multiplicity of $L_\mathtt{j}$ as a composition factor of $X$ for every $\mathtt{j}\in Q_0'$. Recall that if $X$ is a submodule of $I_\mathtt{i}$ for $\mathtt{i}\in Q_0'$, then $e X$ is a submodule of $e I_\mathtt{i}=I_\mathtt{i}^{eAe}$. Similarly, if $X$ is a quotient of $P_\mathtt{i}$ for $\mathtt{i}\in Q_0'$, then $eX$ is a quotient of $e P_\mathtt{i}=P_\mathtt{i}^{eAe}$.
\end{proof}

Suppose now that $(A,Q_0, \unlhd)$ is left standardly stratified and let $e\in A$ be an idempotent supported in some subset $Q_0'$ of $Q_0$. The remainder of this subsection deals with the following question: 
\begin{displayquote}
What does it mean for the idempotent $e$ to be compatible with the left standardly stratified structure of $(A,Q_0, \unlhd)$?
\end{displayquote}
 To answer this, recall that a \emphbf{coideal} $Q_0'$ of a partial order $(Q_0,\unlhd)$ is a subset of $Q_0$ satisfying the following property: if $\mathtt{i}\unlhd \mathtt{j}$ and $\mathtt{i}\in Q_0'$, then $\mathtt{j}\in Q_0'$. Below, there are some candidates for the notion of compatible idempotent. 
\begin{enumerate}
    \item The support $Q_0'$ of $e\in A$ is a coideal of $(Q_0,\unlhd)$. \label{compatible-coideal}
    \item There exists a left standard stratification of $(A,Q_0, \unlhd)$ where the ideal $AeA$ appears, i.e.~there exists a chain of idempotent ideals of $A$
    \[
    0=J_0 \subset \cdots \subset J_n =A
    \]
    with $AeA=J_l$ for some $l \in \{0,1, \ldots,n\}$ and such that $J_k/J_{k-1}$ is a direct sum of standard modules for every $k \in \{1,\ldots,n\}$.  
    \label{compatible-stratifying}
    \item The support $Q_0'$ of $e\in A$ is a coideal of $(Q_0,\unlhd_{\ess})$, where $(Q_0,\unlhd_{\ess})$ is the essential order associated to $(A, Q_0, \unlhd)$.\label{compatible-essential-coideal}
    \item The $A$-module $A/AeA$ has a $\Delta$-filtration.\label{compatible-A/AeA}
    \item The $A$-module $D(A/AeA)$ has a $\bar{\nabla}$-filtration.\label{compatible-D(A/AeA)}
    \item Both $(A/AeA, Q_0 \setminus Q_0', \unlhd)$ and $(eAe, Q_0', \unlhd)$ are left standardly stratified algebras.\label{compatible-quotients-subalgebras}
\end{enumerate}
The next two remarks discuss and clarify the notion of a left standard stratification mentioned in \eqref{compatible-stratifying}.
\begin{rmk}
\label{rmk:standardstrat1}
    Let $A$ be an algebra and $Q_0$ be a set labelling its simples. According to \cite[(2.1.4), Remark 2.1.2 (b)]{CPS96}, a left standard stratification of $A$ is, by definition, a chain of idempotent ideals of $A$
    \[
         0=J_0 \subset \cdots \subset J_n =A
    \]
    such that $J_k/J_{k-1}$ is a projective $A/J_{k-1}$-module for every $k \in \{1,\ldots,n\}$. By  \cite[Theorem 2.2.3]{CPS96}, $A$ has a left standard stratification of length $|Q_0|$ exactly when $A$ is left standardly stratified with respect to some poset structure on $Q_0$. In this case, the subquotients $J_k/J_{k-1}$ of the left standard stratification are direct sums of standard modules (see \cite[Remark 2.2.4]{CPS96}).
   \end{rmk}
   \begin{rmk}\label{rmk:standardstrat2}
   Let $(A, Q_0, \unlhd)$ be a left standardly stratified algebra. Any chain of idempotent ideals  \[
         0=J_0 \subset \cdots \subset J_n =A
    \]
    such that $J_k/J_{k-1}$ is a direct sum of standard modules for every $k \in \{1,\ldots,n\}$ is a left standard stratification. To see this, write $J_1=AeA$, where $e$ is an idempotent supported in a subset $Q_0'$ of $Q_0$. Then $AeA$ is isomorphic to $(\bigoplus_{\mathtt{i}\in Q_0'} P_{\mathtt{i}}^{\oplus[\Kopf(A):L_{\mathtt{i}}]})\oplus X$ as an $A$-module, where $X$ has no projective direct summands and $\Kopf(X)$ is a direct sum of simples with labels in $Q_0'$. Since $J_1$ is a direct sum of standard modules, we must actually have $J_1/J_0\cong AeA\cong \bigoplus_{\mathtt{i}\in Q_0'}P_{\mathtt{i}}^{\oplus[\Kopf(A):L_{\mathtt{i}}]}$ comparing tops and using the Krull--Remak--Schmidt Theorem. The claim follows now by induction on the length of the chain. The left standard stratifications of $A$ giving rise to the same left standardly stratified structure as that of $(A, Q_0,\unlhd)$ are the chains of idempotent ideals of $A$ whose subquotients are direct sums of standard modules over $(A, Q_0,\unlhd)$. 
\end{rmk}

The following theorem relates the candidates for the notion of compatible idempotent listed above. In the case of quasi-hereditary algebras, a recursive approach via idempotents has been considered in \cite[Theorem 1]{DR89} and in \cite[Theorem 2.1]{GKP19}.

\begin{thm}
\label{thm:defofcompatibility}
Let $(A,Q_0, \unlhd)$ be a left standardly stratified algebra and let $e\in A$ be an idempotent supported in $Q_0'$. We have the following chain of implications:
    
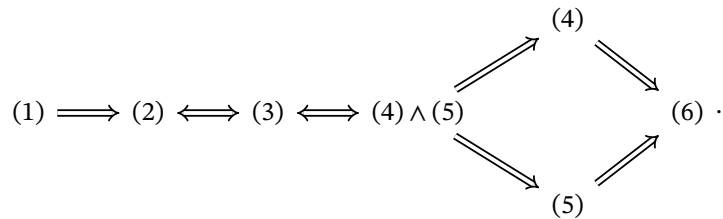
\[
    \begin{tikzcd}
        & & & & \eqref{compatible-A/AeA} \ar[rd, Rightarrow] & \\
        \eqref{compatible-coideal} \ar[r, Rightarrow] & \eqref{compatible-stratifying} \ar[r, Leftrightarrow] & \eqref{compatible-essential-coideal} \ar[r, Leftrightarrow] & \eqref{compatible-A/AeA} \wedge \eqref{compatible-D(A/AeA)}
            \ar[ru,Rightarrow] \ar[rd, Rightarrow] & & \eqref{compatible-quotients-subalgebras} \\
       & & & & \eqref{compatible-D(A/AeA)} \arrow[ru, Rightarrow] &
    \end{tikzcd}.
    \]
\end{thm}

We will prove this in several steps.

\begin{lem}
\label{lem:comp-second}
    Let $(A,Q_0, \unlhd)$ be a left standardly stratified algebra and let $e\in A$ be an idempotent supported in $Q_0'$. If the $A$-module $A/AeA$ has a $\Delta$-filtration, then $(A/AeA, Q_0\setminus Q_0', \unlhd)$ is left standardly stratified, $i (\Delta_\mathtt{i}^{A/AeA})=\Delta_\mathtt{i}$ for every $\mathtt{i}\in Q_0\setminus Q_0'$ and $i$ restricts to an equivalence of exact categories $\mathcal{F}(\Delta^{A/AeA}) \to \mathcal{F}(\{\Delta_{\mathtt{i}}\mid\mathtt{i}\in Q_0\setminus Q_0'\})$.
\end{lem}

\begin{proof}
    The standard modules appearing in a $\Delta$-filtration of $A/AeA$ are such that all their composition factors are of the form $L_\mathtt{i}$ with $\mathtt{i}\in Q_0\setminus Q_0'$. As $\mathcal{F}(\Delta)$ is closed under direct summands (see \cite[Lemma 1.4]{DR92}) and $P_\mathtt{i}/AeP_\mathtt{i}$ is a direct summand of $A/AeA$, $P_\mathtt{i}/AeP_\mathtt{i}$ is $\Delta$-filtered and $\Delta_\mathtt{i}$ must coincide with $\Delta_\mathtt{i}^{A/AeA}$ for $\mathtt{i}\in Q_0\setminus Q_0'$. Moreover, as $\mathcal{F}(\Delta)$ is closed under kernels of epimorphisms (see \cite[Lemma 1.5]{DR92}), $AeP_\mathtt{i}$ is also $\Delta$-filtered as the kernel of $P_{\mathtt{i}} \twoheadrightarrow P_{\mathtt{i}}/AeP_\mathtt{i}$.
    A $\Delta$-filtration of $P_\mathtt{i}$ can be obtained by splicing the $\Delta$-filtration of $AeP_\mathtt{i}$ together with the one of $P_\mathtt{i}/AeP_\mathtt{i}$. In particular, the standard modules appearing in the kernel of the morphism $P_\mathtt{i}/AeP_\mathtt{i} \twoheadrightarrow \Delta^{A/AeA}_\mathtt{i}$ must be all of the form $\Delta^{A/AeA}_\mathtt{j}$ with $\mathtt{j}\rhd \mathtt{i}$. By regarding $\Modu{A/AeA}$ as a Serre subcategory of $\Modu{A}$ it is clear that $i$ restricts to the desired equivalence of exact categories.
\end{proof}

\begin{rmk}
\label{rmk:second}
    Observe that the proof of Lemma \ref{lem:comp-second} can be adapted to prove the following statement for a right standardly stratified algebra $(A,Q_0,\unlhd)$: if $A/AeA$ has a $\bar{\Delta}$-filtration, then $(A/AeA, Q_0\setminus Q_0', \unlhd)$ is right standardly stratified with $i (\bar{\Delta}_\mathtt{i}^{A/AeA})=\bar{\Delta}_\mathtt{i}$ for every $\mathtt{i}\in Q_0\setminus Q_0'$. For this, note that the category $\mathcal{F}(\bar{\Delta})$ is also closed under direct summands and kernels of epimorphisms (\cite[Theorem~3.1, Lemma~3.2]{ADL98}).
\end{rmk}

\begin{lem}
\label{lem:comp-second'}
    Let $(A,Q_0, \unlhd)$ be a left standardly stratified algebra and let $e\in A$ be an idempotent supported in $Q_0'$. If the $A$-module $D(A/AeA)$ has a $\bar{\nabla}$-filtration, then $(A/AeA, Q_0\setminus Q_0', \unlhd)$ is left standardly stratified, $i (\bar{\nabla}_\mathtt{i}^{A/AeA})=\bar{\nabla}_\mathtt{i}$ for every $\mathtt{i}\in Q_0\setminus Q_0'$ and $i$ restricts to an equivalence of exact categories $\mathcal{F}(\bar{\nabla}^{A/AeA}) \to \mathcal{F}(\{\bar{\nabla}_{\mathtt{i}}\mid\mathtt{i}\in Q_0\setminus Q_0'\})$.
\end{lem}
\begin{proof}
Note that the $A$-module $D(A/AeA)$ has a $\bar{\nabla}$-filtration if and only if $A/AeA$, considered as an $A\Op$-module, has a filtration by proper standard modules $\bar{\Delta}^{A\Op}_\mathtt{i}$. Recall that $(A\Op, Q_0, \unlhd)$ is right standardly stratified since $(A, Q_0, \unlhd)$ is left standardly stratified. By Remark \ref{rmk:second}, $((A/A eA)\Op, Q_0\setminus Q_0', \unlhd)$ is right standardly stratified and $\bar{\Delta}_\mathtt{i}^{(A/AeA)\Op}$ coincides with $\bar{\Delta}^{A\Op}_\mathtt{i}$ for every $\mathtt{i}\in Q_0\setminus Q_0'$. The result follows by dualisation. 
\end{proof}

The next example shows that conditions \eqref{compatible-A/AeA} and \eqref{compatible-D(A/AeA)} do not imply each other.

\begin{ex}
\label{ex:first}
    Consider the algebra $A=\Bbbk Q$, where
    \[
    Q=
    \begin{tikzcd}[ column sep = small]
        \overset{\mathtt{1}}{\circ} \ar[r] & \overset{\mathtt{2}}{\circ} & \overset{\mathtt{1}'}{\circ} \ar[l]
    \end{tikzcd}.
    \]
    Endow $Q_0$ with the partial order given by $\mathtt{1} \unlhd \mathtt{2}$, $\mathtt{1}'\unlhd \mathtt{2}$. The algebra $(A,Q_0,\unlhd)$ is quasi-hereditary with simple standard modules and $(Q_0,\unlhd)$ is the associated essential order. Consider the idempotent $e_\mathtt{1}$. Any $A$-module is $\Delta$-filtered, so in particular $A/Ae_\mathtt{1}A$ lies in $\mathcal{F}(\Delta)$, but $\{\mathtt{1}\}$ is not a coideal of $(Q_0,\unlhd)$. The algebra $A/Ae_\mathtt{1}A$ is clearly quasi-hereditary with respect to $(\{\mathtt{1}',\mathtt{2}\},\unlhd)$ and has simple standard modules (cf.~Lemma \ref{lem:comp-second}). However, we have $D(A/Ae_\mathtt{1}A)\cong P_{\mathtt{1}'} \oplus L_{\mathtt{1}'}$ as $A$-modules, and it is easy to check that $D(A/Ae_\mathtt{1}A)$ is not filtered by (proper) costandard modules. It follows that \eqref{compatible-A/AeA} does not imply \eqref{compatible-D(A/AeA)}. By considering the quasi-hereditary algebra $(A\Op,Q_0,\unlhd)$ instead, one concludes that \eqref{compatible-D(A/AeA)}  does not imply \eqref{compatible-A/AeA}. Recall that the trace of $X$ in $Y$, denoted by $\Tra{X}{Y}$, is the smallest submodule of $Y$ generated by $X$. Note that
    \[
        \nabla_\mathtt{2}^{A/Ae_\mathtt{1} A}=
        \begin{tikzcd}[row sep = small]
            \mathtt{1}' \ar[d, no head] \\
            \mathtt{2}
        \end{tikzcd}
        =\bar{\nabla}_\mathtt{2}^{A/Ae_\mathtt{1} A}=\Tra{A/Ae_\mathtt{1} A}{\bar{\nabla}_\mathtt{2}} \subset \bar{\nabla}_\mathtt{2} =
        \begin{tikzcd}[column sep = tiny, row sep = small]
            \mathtt{1} \ar[rd, no head] & & \mathtt{1}'\\
            & \mathtt{2} \ar[ru, no head] &
        \end{tikzcd}
        =\nabla_\mathtt{2}
    \]
    (cf.~Lemmas \ref{lem:comp-first} and \ref{lem:comp-second'}). Notice that $e_\mathtt{1} A e_\mathtt{1}$ is clearly quasi-hereditary with $\Delta_\mathtt{1}^{e_\mathtt{1} A e_\mathtt{1}}=e_\mathtt{1} \Delta_\mathtt{1}$ and $\nabla_\mathtt{1}^{e_\mathtt{1} A e_\mathtt{1}}=e_\mathtt{1} \nabla_\mathtt{1}$.
\end{ex}

The next example shows that the requirement that $i (\Delta_\mathtt{i}^{A/AeA})=\Delta_\mathtt{i}$ for every $\mathtt{i}\in Q_0\setminus Q_0'$ does not imply that $A/AeA$ is $\Delta$-filtered, nor do the identities $i (\bar{\nabla}_\mathtt{i}^{A/AeA})=\bar{\nabla}_\mathtt{i}$ for every $\mathtt{i}\in Q_0\setminus Q_0'$ imply that $D(A/AeA)$ is $\bar{\nabla}$-filtered. The algebra in the example is borrowed from \cite[Example 2.3]{Koe95}, where it was used to demonstrate that not all quasi-hereditary algebras have an exact Borel subalgebra.
\begin{ex}
\label{ex:contraexample_steffen}
    Consider the algebra $A=\Bbbk Q/(\beta \alpha - \delta \gamma)$, where
    \[
    Q=
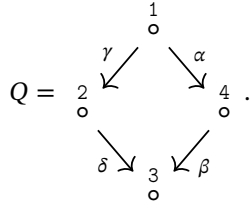
    \begin{tikzcd}[row sep = small, column sep = small]
        & \overset{\mathtt{1}}{\circ} \ar[rd, "\alpha"] \ar[ld, "\gamma"'] & \\
        \overset{\mathtt{2}}{\circ} \ar[rd, "\delta"'] & & \overset{\mathtt{4}}{\circ} \ar[ld, "\beta"] \\
        & \overset{\mathtt{3}}{\circ} &
    \end{tikzcd}.
    \]
    Note that $A$ is quasi-hereditary with respect to the natural order on the set $\{\mathtt{1},\mathtt{2},\mathtt{3},\mathtt{4}\}$. Consider the idempotent $e_\mathtt{2}$. It is not hard to check that $\Delta_\mathtt{i}=i (\Delta_\mathtt{i}^{A/Ae_\mathtt{2} A})$ for $\mathtt{i}\in \{\mathtt{1},\mathtt{3},\mathtt{4}\}$. However, $A/Ae_\mathtt{2} A$ is not $\Delta$-filtered.
\end{ex}

Given a left standardly stratified algebra $(A,Q_0, \unlhd)$ and a module $X$ in $\mathcal{F}(\Delta)$, denote the multiplicity of $\Delta_\mathtt{i}$ in $X$ by $(X:\Delta_\mathtt{i})$, and define $(Y:\bar{\nabla}_\mathtt{i})$ analogously for $Y$ in $\mathcal{F}(\bar{\nabla})$. Consider a similar notation for the multiplicities of proper standard modules and of costandard modules when the underlying algebra is right standardly stratified. For applications of the next two results it is useful to remember that BGG reciprocity holds for standardly stratified algebras (see \cite[Theorem 2.5]{ADL98}). In particular, given a left standardly stratified algebra $(A,Q_0, \unlhd)$ and $\mathtt{i},\mathtt{j}\in Q_0$, we have that $[\Delta_\mathtt{j}:L_\mathtt{i}]\neq 0$ if and only if $(I_\mathtt{i}:\bar{\nabla}_\mathtt{j})\neq 0$ and also $[\bar{\nabla}_\mathtt{j}:L_\mathtt{i}]\neq 0$ if and only if $(P_\mathtt{i}:\Delta_\mathtt{j})\neq 0$.

\begin{lem}
\label{lem:comp-third}
    Let $(A,Q_0, \unlhd)$ be a left standardly stratified algebra and let $e\in A$ be an idempotent supported in $Q_0'$. Suppose that the $A$-module $A/AeA$ has a $\Delta$-filtration and let $\mathtt{i},\mathtt{j} \in Q_0$. If $[\Delta_\mathtt{j}:L_\mathtt{i}]\neq 0$ and $\mathtt{i}\in Q_0'$, then $\mathtt{j}\in Q_0'$.
\end{lem}
\begin{proof}
   Suppose that $[\Delta_\mathtt{j}:L_\mathtt{i}]\neq 0$ and $\mathtt{i}\in Q_0'$. If we had $\mathtt{j}\in Q_0 \setminus Q_0'$, then $\Delta_\mathtt{j}=i(\Delta^{A/AeA}_\mathtt{j})$ by Lemma \ref{lem:comp-second} and this would lead to a contradiction as $\Delta_\mathtt{j}$ would be an $A/AeA$-module with a composition factor labelled by an element of $Q_0'$. Hence $\mathtt{j}\in Q_0'$.
\end{proof}

\begin{lem}
\label{lem:comp-third'}
    Let $(A,Q_0, \unlhd)$ be a left standardly stratified algebra and let $e\in A$ be an idempotent supported in $Q_0'$. Suppose that the $A$-module $D(A/AeA)$ has a filtration by proper costandard modules and let $\mathtt{i},\mathtt{j} \in Q_0$. If $[\bar{\nabla}_\mathtt{j}:L_\mathtt{i}]\neq 0$ and $\mathtt{i}\in Q_0'$, then $\mathtt{j}\in Q_0'$.
\end{lem}
\begin{proof}
    Dualise the argument in the proof of Lemma \ref{lem:comp-third} and use Lemma \ref{lem:comp-second'}.
\end{proof}

\begin{lem}
\label{lem:comp-fourth}
   Let $(A,Q_0, \unlhd)$ be a left standardly stratified algebra and let $e\in A$ be an idempotent supported in $Q_0'$. If the $A$-module $A/AeA$ has a $\Delta$-filtration, then $(eAe, Q_0', \unlhd)$ is left standardly stratified and $\Hom_{eAe}(eA, \bar{\nabla}_\mathtt{i}^{eAe})=\bar{\nabla}_\mathtt{i}$ for every $\mathtt{i}\in Q_0'$. Moreover, $\Delta_\mathtt{i}^{eAe}=e\Delta_\mathtt{i}$ and $\bar{\nabla}_\mathtt{i}^{eAe}=e \bar{\nabla}_\mathtt{i}$ for every $\mathtt{i}\in Q_0'$. 
\end{lem}

\begin{proof}
    Let $\mathtt{j} \in Q_0$. Note that $e \Delta_\mathtt{j}\neq 0$ if and only if $\Hom_A(Ae, \Delta_\mathtt{j})\neq 0$, and this happens exactly when $[\Delta_\mathtt{j}:L_\mathtt{i}]\neq 0$ for some $\mathtt{i}\in Q_0'$. By Lemma \ref{lem:comp-third}, the latter situation holds exactly when $\mathtt{j}\in Q_0'$.

    We claim that $(eAe, Q_0', \unlhd)$ is left standardly stratified with $e\Delta_\mathtt{i}=\Delta_\mathtt{i}^{eAe}$ for every $\mathtt{i} \in Q_0'$. Lemma \ref{lem:comp-first} implies that $e\Delta_\mathtt{i}$ is a quotient of $\Delta_\mathtt{i}^{eAe}$.  Fix  $\mathtt{i} \in Q_0'$ and consider the short exact sequence
    \begin{equation}
    \label{eq:standardses}
    \begin{tikzcd}
        0 \ar[r] & U_\mathtt{i} \ar[r] & P_\mathtt{i} \ar[r] & \Delta_\mathtt{i} \ar[r] & 0
    \end{tikzcd},
    \end{equation}
    where $U_\mathtt{i}$ is filtered by modules $\Delta_\mathtt{j}$ with $\mathtt{j} \rhd \mathtt{i}$. By the considerations in the first paragraph of the proof, the (exact) functor $e(-): \Modu{A} \to \Modu{eAe}$ maps \eqref{eq:standardses} to the short exact sequence
    \[
    \begin{tikzcd}
        0 \ar[r] & eU_\mathtt{i} \ar[r] & P_\mathtt{i}^{eAe} \ar[r] & e\Delta_\mathtt{i} \ar[r] & 0
    \end{tikzcd},
    \]
    where $eU_\mathtt{i}$ is filtered by $e\Delta_\mathtt{j}$ with $\mathtt{j} \rhd \mathtt{i}$ and $\mathtt{j} \in Q_0'$. Consequently, $e\Delta_\mathtt{i}$ cannot be a proper quotient of $\Delta_\mathtt{i}^{A/AeA}$ as all the composition factors in the top of $eU_\mathtt{i}$ have a label strictly greater than $\mathtt{i}$. Hence $\Delta_\mathtt{i}^{A/AeA}=e\Delta_\mathtt{i}$ and $(eAe, Q_0', \unlhd)$ is left standardly stratified. 

    As before, fix $\mathtt{i} \in Q_0'$. Note that $e\bar{\nabla}_\mathtt{i}$ is a non-zero submodule of $\bar{\nabla}_\mathtt{i}^{eAe}$ by Lemma \ref{lem:comp-first}. Lemma \ref{lem:comp-third} (together with BGG reciprocity) implies that $\bar{\nabla}_\mathtt{i}$ has an injective copresentation of the form
    \begin{equation}
    \label{eq:propercostandardses}
    \begin{tikzcd}
        0 \ar[r] & \bar{\nabla}_\mathtt{i} \ar[r] & I_\mathtt{i} \ar[r, "f_\mathtt{i}"] & \bigoplus_{\substack{\mathtt{j} \in Q_0' \\ \mathtt{j} \unrhd \mathtt{i}}} I_\mathtt{j}^{\oplus b_\mathtt{j}}
    \end{tikzcd}.
    \end{equation}
    By applying $e(-): \Modu{A} \to \Modu{eAe}$ to \eqref{eq:propercostandardses}, we conclude that the simples appearing in the socle of $I_\mathtt{i}^{eAe}/e\bar{\nabla}_\mathtt{i}$ are of the form $L_\mathtt{j}^{eAe}$ with $\mathtt{j} \in Q_0'$ and $\mathtt{j} \unrhd \mathtt{i}$, so we must have $\bar{\nabla}_\mathtt{i}^{eAe}=e\bar{\nabla}_\mathtt{i}$. Finally, observe that composition $\Hom_{eAe}(eA,-) \circ e(-): \Modu{A} \to \Modu{A}$ maps \eqref{eq:propercostandardses} to the exact sequence
    \[
    \begin{tikzcd}
        0 \ar[r] & \Hom_{eAe}(eA,\bar{\nabla}^{eAe}_\mathtt{i}) \ar[r] & I_\mathtt{i} \ar[r, "f_\mathtt{i}"] & \bigoplus_{\substack{\mathtt{j} \in Q_0' \\ \mathtt{j} \unrhd \mathtt{i}}} I_\mathtt{j}^{\oplus b_\mathtt{j}}
    \end{tikzcd}.
    \]
    Thus $\Hom_{eAe}(eA,\bar{\nabla}^{eAe}_\mathtt{i})=\bar{\nabla}_\mathtt{i}$. \end{proof}

\begin{lem}
\label{lem:comp-fourth'}
    Let $(A,Q_0, \unlhd)$ be a left standardly stratified algebra and let $e\in A$ be an idempotent supported in $Q_0'$. If the $A$-module $D(A/AeA)$ has a $\bar{\nabla}$-filtration, then $(eAe, Q_0', \unlhd)$ is left standardly stratified and $Ae\otimes_{eAe} \Delta_\mathtt{i}^{eAe}= \Delta_\mathtt{i}$ for every $\mathtt{i}\in Q_0'$. Moreover, $\Delta_\mathtt{i}^{eAe}=e\Delta_\mathtt{i}$ and $\bar{\nabla}_\mathtt{i}^{eAe}=e \bar{\nabla}_\mathtt{i}$ for every $\mathtt{i}\in Q_0'$.
\end{lem}

\begin{proof}
    This follows by dualisation.
\end{proof}

\begin{lem}
\label{lem:comp-fifth}
    Let $(A,Q_0, \unlhd)$ be a left standardly stratified algebra and let $e\in A$ be an idempotent supported in $Q_0'$. If the $A$-module $A/AeA$ has a $\Delta$-filtration and the $A$-module $D(A/AeA)$ has a $\bar{\nabla}$-filtration, then $Q_0'$ is a coideal of the essential partial order $(Q_0, \unlhd_{\ess})$ associated to $(A, Q_0, \unlhd)$. 
\end{lem}

\begin{proof}
Let $\mathtt{i},\mathtt{j} \in Q_0$ with $\mathtt{i}\unlhd_{\ess} \mathtt{j}$. Suppose that $\mathtt{i} \in Q_0'$. By definition, there is a chain 
\[\mathtt{i}=\mathtt{i}_0 \unlhd_{\ess} \cdots \unlhd_{\ess} \mathtt{i}_n =\mathtt{j},\] where for $k\in \{1, \ldots,n\}$ either $[\Delta_{\mathtt{i}_k}:L_{\mathtt{i}_{k-1}}]\neq 0$ or $[\bar{\nabla}_{\mathtt{i}_k}:L_{\mathtt{i}_{k-1}}]\neq 0$. Using Lemmas \ref{lem:comp-third} and \ref{lem:comp-third'} and induction, it follows that $\mathtt{j} \in Q_0'$.
\end{proof}

\begin{proof}[Proof of Theorem \ref{thm:defofcompatibility}]
By \cite[Theorem 2.2.6]{CPS96} and \cite[Lemma 1.3]{AHLU00}, condition \eqref{compatible-essential-coideal} implies that $(A/AeA,Q_0\setminus Q_0', \unlhd)$ is left standardly stratified with $\Delta^{A/AeA}_\mathtt{i}=\Delta_\mathtt{i}$ and $\bar{\nabla}^{A/AeA}_\mathtt{i}=\bar{\nabla}_\mathtt{i}$ for all $\mathtt{i} \in Q_0 \setminus Q_0'$. In particular, \eqref{compatible-essential-coideal} implies both \eqref{compatible-A/AeA} and \eqref{compatible-D(A/AeA)}. By Lemma \ref{lem:comp-fifth}, if both \eqref{compatible-A/AeA} and \eqref{compatible-D(A/AeA)} hold, then \eqref{compatible-essential-coideal} also holds. By Lemmas \ref{lem:comp-second} and \ref{lem:comp-fourth}, \eqref{compatible-A/AeA} implies \eqref{compatible-quotients-subalgebras}. According to Lemmas \ref{lem:comp-second'} and \ref{lem:comp-fourth'}, \eqref{compatible-D(A/AeA)} implies \eqref{compatible-quotients-subalgebras}. To show simultaneously that both \eqref{compatible-coideal} and \eqref{compatible-essential-coideal} imply \eqref{compatible-stratifying}, consider any refinement $(Q_0, \unlhd')$ of the essential order of $(A,Q_0, \unlhd)$ and suppose that the support $Q_0'$ of the idempotent $e\in A$ is a coideal of $(Q_0, \unlhd')$. It is possible to refine $(Q_0, \unlhd')$ to a total order, so that $Q_0'$ remains a coideal of the total order. It then follows from \cite[Remark 2.2.4]{CPS96}  that there exists a left standard stratification of $(A,Q_0, \unlhd')$ of length $n=|Q_0|$ containing an (idempotent) ideal $J_l$ with $l=|Q_0'|$, such that each subquotient $J_k/J_{k-1}$ is a non-zero direct sum of standard modules $\Delta_{\mathtt{i}_k}$ with $\mathtt{i}_k\in Q_0$ and $\mathtt{i}_k\neq \mathtt{i}_{k'}$ for distinct $k$ and 
$k'$, and moreover $\mathtt{i}_k\in Q_0'$ for $k\in \{1, \cdots, l\}$. By applying  \cite[Lemma 2.1.5]{CPS96} to the sequence of idempotent ideals $0\subset J_l \subset \cdots \subset J_n=A$, we conclude that the top of $J_l$ must be a direct sum of modules $L_\mathtt{i}$, $\mathtt{i}\in Q_0'$, each with multiplicity at least one. Hence, $J_l=AeA$. It remains to show that \eqref{compatible-stratifying} implies \eqref{compatible-essential-coideal}. Consider $\mathtt{j}\in Q_0'$ and let $\mathtt{i}\in Q_0$ be such that $[\Delta_\mathtt{i}:L_\mathtt{j}]\neq 0$ or $[\bar{\nabla}_\mathtt{i}:L_\mathtt{j}]\neq 0$. It suffices to show that $\mathtt{i}\in Q_0'$. If  $[\Delta_\mathtt{i}:L_\mathtt{j}]\neq 0$, Lemma \ref{lem:comp-third} implies that $\mathtt{i}\in Q_0'$. Otherwise, if $[\bar{\nabla}_\mathtt{i}:L_\mathtt{j}]\neq 0$, by BGG reciprocity, $\Delta_\mathtt{i}$ appears in a $\Delta$-filtration of $P_\mathtt{j}$ and consequently in a $\Delta$-filtration of $AeA=J_l$. Therefore, $\Delta_\mathtt{i}$ is a direct summand of $J_k/J_{k-1}$ for some $k\leq l$ and \cite[Lemma 2.1.5]{CPS96} implies that $L_\mathtt{i}$ is a composition factor of $\Kopf(J_l)$. Thus $\mathtt{i}\in Q_0'$. 
\end{proof}

In Example \ref{ex:first}, we have already seen that \eqref{compatible-A/AeA} and \eqref{compatible-D(A/AeA)} do not imply each other. It is also clear that \eqref{compatible-essential-coideal} does not generally imply \eqref{compatible-coideal}. The next example shows that the remaining one-sided implications in the statement of Theorem \ref{thm:defofcompatibility} are not equivalences, i.e.~\eqref{compatible-quotients-subalgebras} neither implies \eqref{compatible-A/AeA} nor \eqref{compatible-D(A/AeA)}. 

\begin{ex}
\label{ex:forref}
    Consider the algebra $A=\Bbbk Q/(\alpha \beta)$ where
    \[
    Q=
    \begin{tikzcd}[ column sep = small]
        \overset{\mathtt{1}}{\circ} \ar[r, bend left, "\alpha"]  & \overset{\mathtt{2}}{\circ} \ar[l, bend left, "\beta"]
    \end{tikzcd}.
    \]
    Note that $A$ is quasi-hereditary with respect to the natural order on $\{\mathtt{1},\mathtt{2}\}$, in particular it is left standardly stratified. In fact, it corresponds to the principal block of BGG category $\mathcal{O}$ of $\mathfrak{sl}_2$, see e.g. \cite[Section 5.1.1]{Stro03}. Both $A/Ae_\mathtt{1} A$ and $e_\mathtt{1} A e_\mathtt{1}$ are left standardly stratified algebras. However, $\Delta_\mathtt{2}^{A/Ae_\mathtt{1}A}$ is a proper quotient of $\Delta_\mathtt{2}$ and $\bar{\nabla}_\mathtt{2}^{A/Ae_\mathtt{1} A}$ is a proper submodule of $\bar{\nabla}_\mathtt{2}$. Neither $A/Ae_\mathtt{1}A$ is $\Delta$-filtered nor $D(A/Ae_\mathtt{1} A)$ is $\bar{\nabla}$-filtered.
\end{ex}
\begin{rmk}
    As illustrated by the previous example, the fact that an algebra $A$ is left standardly stratified with respect to $(Q_0,\unlhd)$ and satisfies condition \eqref{compatible-quotients-subalgebras} does not guarantee that the associated idempotent ideal appears in a left standard stratification of $(A, Q_0,\unlhd)$.
    
    On a related note, observe that the standardly stratified property is not, in general, inherited from both a quotient and a corner algebra. More precisely, the conditions that $A/AeA$ and $eAe$ are both left (resp.~right) standardly stratified do not imply that $A$ itself is left (resp.~right) standardly stratified. To see this, let $Q$ be the quiver from Example \ref{ex:forref} and consider the algebra $A=\Bbbk Q/(\rad{\Bbbk Q})^2$. One can check that $A$ is not left (nor right) standardly stratified with respect to any poset structure on the set of simple modules. Nevertheless, both $e_{\mathtt{1}}Ae_{\mathtt{1}}$ and $A/Ae_{\mathtt{1}}A$ are quasi-hereditary.
    
    Observe that \cite[Theorem 1]{DR89} and \cite[Theorem 2.1]{GKP19} provide additional hypotheses under which quasi-hereditary structures on $eAe$ and $A/AeA$ can be lifted to a quasi-hereditary structure on $A$, with the idempontent ideal $AeA$ appearing in an associated heredity chain. It is natural to expect that analogous results can be established for left (and right) standardly stratified algebras.
\end{rmk}

Example \ref{ex:last} shows that the fact that $A/AeA$ is filtered by $\Delta^{A/AeA}_\mathtt{i}$ with $\mathtt{i}\in Q_0\setminus Q_0'$ does not imply \eqref{compatible-A/AeA}. Similarly, there exist examples where $D(A/AeA)$ is filtered by $\bar{\nabla}^{A/AeA}_\mathtt{i}$ with $\mathtt{i}\in Q_0\setminus Q_0'$ and for which \eqref{compatible-D(A/AeA)} fails.

\begin{ex}
\label{ex:last}
    Consider the algebra $A=\Bbbk Q/(\alpha_1 \beta_1 - \beta_2\alpha_2, \alpha_2 \beta_2)$ where
    \[
    Q=
    \begin{tikzcd}[ column sep = small]
        \overset{\mathtt{1}}{\circ} \ar[r, bend left, "\alpha_1"]  & \overset{\mathtt{2}}{\circ} \ar[l, bend left, "\beta_1"] \ar[r, bend left, "\alpha_2"] & \overset{\mathtt{3}}{\circ} \ar[l, bend left, "\beta_2"]
    \end{tikzcd}.
    \]
    This is a quasi-hereditary algebra for the usual order on $\{\mathtt{1},\mathtt{2},\mathtt{3}\}$. It is the Auslander algebra of $\Bbbk [x]/(x^3)$.   The $A$-module $A/A(e_\mathtt{1}+e_\mathtt{2})A$ is isomorphic to $L_\mathtt{3}$, so it is obviously filtered by $\Delta_\mathtt{3}^{A/A(e_\mathtt{1}+e_\mathtt{2})A}$. However, $(e_\mathtt{1}+e_\mathtt{2})A(e_\mathtt{1}+e_\mathtt{2})$ is not left standardly stratified.
\end{ex}

\begin{defn}
\label{defn:compatibility}
    An idempotent $e\in A$ is \emphbf{compatible} with a left standardly stratified algebra $(A, Q_0,\unlhd)$ if $A/AeA$ is $\Delta$-filtered and $D(A/AeA)$ is $\bar{\nabla}$-filtered.
\end{defn}

 Theorem \ref{thm:defofcompatibility} provides alternative characterisations of an idempotent compatible with a left standardly stratified algebra. Naturally, Theorem \ref{thm:defofcompatibility} and Definition \ref{defn:compatibility} can be adapted to right standardly stratified algebras by using right modules in condition \eqref{compatible-stratifying} and replacing $\Delta$ by $\bar{\Delta}$ and $\bar{\nabla}$ by $\nabla$ elsewhere. If an idempotent $e\in A$ is compatible with a left (or right) standardly stratified algebra $(A, Q_0,\unlhd)$, then $AeA$ is an $m$-idempotent ideal in the sense of \cite[Section 1]{APT92} for every $m\in \N$. This follows, for instance, from \cite[Corollary 6.2]{APT92} and Remark \ref{rmk:standardstrat2}. In other words, if $e$ is a compatible idempotent, then $AeA$ is a \emphbf{$\infty$-idempotent ideal} or, equivalently said, the embedding $i\colon\Modu{A/AeA}\to \Modu{A}$ is an $\infty$-homological embedding.

\begin{defn}[{\cite[Definition 3.6]{Psa14}}]
	\label{defn:homologicalembedding}
	An exact functor $i\colon \mathcal{A}\rightarrow \mathcal{B}$ between abelian categories is an \emphbf{$m$-homological embedding} (for $m\in \Znn\cup \{\infty\}$) if it induces isomorphisms
	\[
	\Exte{\mathcal{A}}{j}{X}{Y}\longrightarrow \Exte{\mathcal{B}}{j}{i(X)}{i(Y)}
	\]
	for all $X$, $Y$ in $\mathcal{A}$ and all integers $j$ satisfying $0\leq j\leq m$. 
\end{defn}

    Note that $\infty$-idempotent ideals also go under the name of stratifying ideals and were first introduced in \cite[Definition 2.1.1, Remark 2.1.2]{CPS96}.

\begin{rmk}
    Let $e$ be an idempotent in a left standardly stratified algebra $(A, Q_0,\unlhd)$ satisfying \eqref{compatible-A/AeA}. For the embedding $i\colon \Modu{A/AeA}\to \Modu{A}$, the abelian groups $\Exte{A/AeA}{j}{X}{\Hom_A(A/AeA,Y)}$ and $\Exte{A}{j}{i(X)}{Y}$ are isomorphic for every $X$ in $\Modu{A/AeA}$, $Y$ in $\mathcal{F}(\bar{\nabla})$ and $j\geq 0$. Similarly, if \eqref{compatible-D(A/AeA)} holds, then $\Exte{A/AeA}{j}{X/AeX}{Y}\cong \Exte{A}{j}{X}{i(Y)}$ for every $X$ in $\mathcal{F}(\Delta)$, $Y$ in $\Modu{A/AeA}$ and $j\geq 0$. The previous observations follow from \cite[Propositions 1.1, 1.2]{APT92}. However, as shown in the next example, conditions \eqref{compatible-A/AeA} or \eqref{compatible-D(A/AeA)} alone are not enough to guarantee that the ideal $AeA$ is $\infty$-idempotent and not even to assure that the $\Ext$-groups between standard modules or proper costandard modules in $\Modu{A/AeA}$ and $\Modu{A}$ coincide. Ana\-lo\-gous observations hold for the right standardly stratified case.
 \end{rmk}
 
\begin{ex}
    The algebra $A=\Bbbk Q/(\beta \gamma)$ where
    \[
    Q=
    \begin{tikzcd}[ column sep = small]
        \overset{\mathtt{1}}{\circ} \ar[r, "\gamma"] & \overset{\mathtt{2}}{\circ} \ar[r, bend right, "\beta"'] \ar[r, bend left, "\alpha"] & \overset{\mathtt{3}}{\circ} 
    \end{tikzcd}
    \]
    is quasi-hereditary with respect to the natural order on $\{\mathtt{1},\mathtt{2},\mathtt{3}\}$ and has simple standard modules. The $A$-module $A/Ae_\mathtt{2}A$ is obviously $\Delta$-filtered and all the $\Ext$-groups over $A/Ae_\mathtt{2}A$ of positive degree are zero as $A/Ae_\mathtt{2}A$ is semisimple. However, $\Ext_{A}^2(L_\mathtt{1},L_\mathtt{3})\neq 0$.
\end{ex}

Given a recollement as in \eqref{eq:recol} and $m\in \Znn \cup \{\infty\}$, as is done in \cite[Section 2]{APT92} as well as \cite[Section 3.1]{Psa14}, define $\mathcal{I}_m$ as the class of objects $Y$ in $\mathcal{B}$ for which there exists an exact sequence
\[
\begin{tikzcd}
	0 \ar[r] & Y \ar[r] & r(I^0) \ar[r] & \cdots \ar[r] & r(I^m),
\end{tikzcd}
\]
with $I^j$ an injective object in $\mathcal{C}$ for every integer $j$ satisfying $0\leq j \leq m$.

Similarly, given $m\in \Znn \cup \{\infty\}$, let $\mathcal{P}_m$ be the class of objects $X$ in $\mathcal{B}$ for which there exists an exact sequence
\[
\begin{tikzcd}
	l(P_m) \ar[r] & \cdots \ar[r] & l(P_0) \ar[r] & X \ar[r] & 0,
\end{tikzcd}
\]
with $P_j$ a projective object in $\mathcal{C}$ for every integer $j$ satisfying $0\leq j \leq m$.

\begin{lem}
\label{lem:X-and-Y-infinity}
   Let $(A, Q_0,\unlhd)$ be a left (resp.~right) standardly stratified algebra, let $e\in A$ be an idempotent with support $Q_0'$ and consider the associated recollement \eqref{eq:recol-modules}. If the $A$-module $A/AeA$ is $\Delta$-filtered (resp.~$\bar{\Delta}$-filtered) then $\mathcal{F}(\{\bar{\nabla}_{\mathtt{i}}\mid \mathtt{i} \in Q_0'\})\subseteq \mathcal{I}_\infty$ (resp.~$\mathcal{F}(\{\nabla_{\mathtt{i}}\mid \mathtt{i} \in Q_0'\})\subseteq \mathcal{I}_\infty$). If the $A$-module $D(A/AeA)$ is $\bar{\nabla}$-filtered (resp.~$\nabla$-filtered) then $\mathcal{F}(\{\Delta_{\mathtt{i}}\mid \mathtt{i} \in Q_0'\})\subseteq \mathcal{P}_\infty$ (resp.~$\mathcal{F}(\{\bar{\Delta}_{\mathtt{i}}\mid \mathtt{i} \in Q_0'\})\subseteq \mathcal{P}_\infty$).
\end{lem}

\begin{proof}
    Let $(A, Q_0,\unlhd)$ be left standardly stratified and let $e\in A$ be an idempotent supported in $Q_0'$. Assume that $A/AeA$ is $\Delta$-filtered and consider $X$ in $\mathcal{F}(\{\bar{\nabla}_{\mathtt{i}}\mid \mathtt{i} \in Q_0'\})$. Lemma \ref{lem:comp-third} (together with BGG reciprocity) implies that all modules appearing in a minimal injective coresolution of $X$ are a direct sum of injective indecomposables labelled by elements in $Q_0'$, i.e.~a direct sum of $\Hom_{eAe}(eA,I_{\mathtt{i}}^{eAe})$ for $\mathtt{i}\in Q_0'$.  Using the notation $r=\Hom_{eAe}(eA,-)$ from \eqref{eq:recol}, all injectives in a minimal injective coresolution of $X$ are then of the form $r(I)$ for $I$ an injective in $\Modu{eAe}$, which proves the inclusion $\mathcal{F}(\{\bar{\nabla}_{\mathtt{i}}\mid \mathtt{i} \in Q_0'\})\subseteq \mathcal{I}_\infty$. Similarly, if $D(A/AeA)$ is $\bar{\nabla}$-filtered and $X$ lies in $\mathcal{F}(\{\Delta_{\mathtt{i}}\mid \mathtt{i} \in Q_0'\})$, Lemma \ref{lem:comp-third'} implies that all the modules appearing in a minimal projective resolution of $X$ are a direct sum of projectives labelled by elements in $Q_0'$, i.e.~a direct sum of $Ae\otimes_{eAe} P_{\mathtt i}^{eAe}=l(P_{\mathtt i}^{eAe})$ for ${\mathtt i}\in Q_0'$. Hence, the inclusion $\mathcal{F}(\{\Delta_{\mathtt{i}}\mid \mathtt{i} \in Q_0'\})\subseteq \mathcal{P}_\infty$ holds whenever $D(A/AeA)$ is $\bar{\nabla}$-filtered. The right standardly stratified case is similar.
\end{proof}

The next result is an analogue of the equivalences of categories appearing in Lemmas \ref{lem:comp-second} and \ref{lem:comp-second'} and of the corresponding results for the right standardly stratified case.

\begin{cor}
\label{cor:equivalence-eAe}
 Let $(A, Q_0,\unlhd)$ be a left (resp.~right) standardly stratified algebra and let $e\in A$ be an idempotent with support $Q_0'$. Consider the functor $e(-):\Modu{A} \to \Modu{eAe}$. If the $A$-module $A/AeA$ is $\Delta$-filtered (resp.~$\bar{\Delta}$-filtered) then the functor $e(-)$ restricts to an equivalence of exact categories \[\mathcal{F}(\{\bar{\nabla}_\mathtt{i} \mid \mathtt{i} \in Q_0' \})\to \mathcal{F}(\bar{\nabla}^{eAe})\] (resp.~$\mathcal{F}(\{\nabla_\mathtt{i} \mid \mathtt{i} \in Q_0' \})\to \mathcal{F}(\nabla^{eAe})$). If the $A$-module $D(A/AeA)$ is $\bar{\nabla}$-filtered (resp.~$\nabla$-filtered) then the functor $e(-)$ restricts to an equivalence of exact categories \[\mathcal{F}(\{\Delta_\mathtt{i} \mid \mathtt{i} \in Q_0' \})\to \mathcal{F}(\Delta^{eAe})\] (resp.~$\mathcal{F}(\{{\bar{\Delta}}_\mathtt{i} \mid \mathtt{i} \in Q_0' \})\to \mathcal{F}({\bar{\Delta}}^{eAe})$).
\end{cor}

\begin{proof}
     Let $(A, Q_0,\unlhd)$ be left standardly stratified and let $e\in A$ be an idempotent supported in $Q_0'$. Suppose that the $A$-module $A/AeA$ is $\Delta$-filtered. Using Lemma \ref{lem:comp-fourth} and induction on the number of proper costandard modules appearing in a $\bar{\nabla}$-filtration, it is clear that the exact functor $e(-)$ restricts to a functor from  $\mathcal{F}(\{\bar{\nabla}_\mathtt{i} \mid \mathtt{i} \in Q_0' \})$ to $\mathcal{F}(\bar{\nabla}^{eAe})$. By Lemma \ref{lem:X-and-Y-infinity} and \cite[Theorem 3.2]{APT92}, $e(-)$ induces isomorphisms $\Ext_A^n(X,Y) \to \Ext_{eAe}^n(eX,eY)$ for every $n\geq 0$ and every $X$ and $Y$ in $\mathcal{F}(\{\bar{\nabla}_{\mathtt{i}}\mid \mathtt{i} \in Q_0'\})$. In particular, the restriction $\mathcal{F}(\{\bar{\nabla}_\mathtt{i} \mid \mathtt{i} \in Q_0' \})\to \mathcal{F}(\bar{\nabla}^{eAe})$ is fully faithful. To prove that it is essentially surjective, note first that every proper costandard $eAe$-module is isomorphic to $e\bar{\nabla}_{\mathtt i}$ for some $\mathtt{i} \in Q_0'$ and then use induction and the isomorphism of $\Ext$-groups for $n=1$. The case where the $A$-module $D(A/AeA)$ is $\bar{\nabla}$-filtered is dual and follows from Lemmas \ref{lem:comp-fourth'} and \ref{lem:X-and-Y-infinity} and from \cite[Theorem 3.2]{APT92}. The corresponding results about right standardly stratified algebras can be proved in a similar way.
\end{proof}

\section{Exact Borel subalgebras} 
\label{sec:exact-borels}
Adapting from the case of BGG category $\mathcal{O}$, in \cite{Koe95}, Koenig defined the notion of an exact Borel subalgebra of a quasi-hereditary algebra. A slightly different generalisation of his notion to standardly stratified algebras than the  one we consider has been studied in \cite[Definition 1]{KM02}.

Recently, the study of exact Borel subalgebras of standardly stratified algebras has gained more attention and several versions have been investigated, e.g.~in \cite{CZ19, BPS23, Got24}. Over an algebraically closed field, the notion we provide coincides with \cite[Definition 2.3]{Got24}. For examples of exact Borel subalgebras of quasi-hereditary algebras, we refer the reader to the two surveys \cite{Kul17, Kul23}. 

\begin{defn}
\label{defn:exact-Borel}
A (unital) subalgebra $B$ of a left standardly stratified algebra $(A,Q_0,\unlhd)$ is an \emphbf{exact Borel subalgebra} if the following hold:
\begin{enumerate}
    \item\label{borel-cond1} the induction functor $A\otimes_{B}-\colon \Modu{B}\rightarrow\Modu{A}$ is exact;
    \item\label{borel-cond2} $Q_0$ indexes the isomorphism classes of simple modules over $B$ and $B$ is left standardly stratified with respect to $(Q_0, \unlhd)$ with simple standard modules;
    \item\label{borel-cond3} $A\otimes L_\mathtt{i}^{B}\cong \Delta_\mathtt{i}$ for every $\mathtt{i}\in Q_0$;
    \item\label{borel-cond4} $\End_{A}(L_\mathtt{i}^A)$ and $\End_{B}(L_\mathtt{i}^{B})$ are isomorphic as $\Bbbk$-vector spaces for every $\mathtt{i}\in Q_0$.
\end{enumerate}
\end{defn}

Exact Borel subalgebras of right standardly stratified algebras are defined similarly, by using proper standard modules instead of standard modules in \eqref{borel-cond2} and \eqref{borel-cond3}, and by replacing ``left standardly stratified'' by ``right standardly stratified'' in \eqref{borel-cond2}.

\begin{rmk}
Except in \cite{Conde2021}, the definition of an exact Borel subalgebra assumes that the underlying field $\Bbbk$ is algebraically closed. In this case, condition \eqref{borel-cond4} can be omitted since it is automatically satisfied. In general, this is not the case. As an example consider the $\mathbb{R}$-algebra $\mathbb{C}$. Since it is semisimple, it is of course quasi-hereditary. We claim that the $\mathbb{R}$-subalgebra $\mathbb{R}$ of $\mathbb{C}$ satisfies \eqref{borel-cond1}--\eqref{borel-cond3} in the above definition, but obviously not \eqref{borel-cond4}. The induction functor is clearly exact as $\mathbb{R}$ is a field. Both $\mathbb{R}$ and $\mathbb{C}$ have a unique simple module (both are fields), so \eqref{borel-cond2} holds and $\mathbb{C}\otimes_\mathbb{R} \mathbb{R} \cong \mathbb{C}$, so \eqref{borel-cond3} holds. However, we argue that semisimple algebras should only have themselves as exact Borel subalgebras, therefore we add condition \eqref{borel-cond4} in the non-algebraically closed case. 
\end{rmk}

\begin{rmk}
According to Definition \ref{defn:exact-Borel}, exact Borel subalgebras of left standardly stratified algebras are actually (directed) quasi-hereditary algebras (see \cite[Section 2]{Koe96}). It follows that exact Borel subalgebras of left or right standardly stratified algebras are always right standardly stratified with simple proper standard modules and with injective costandard modules.
\end{rmk}

    Note that a quasi-hereditary algebra may be viewed both as a left and as a right standardly stratified algebra, so a priori, one may consider two types of exact Borel subalgebras which do not necessarily coincide. It is also possible that $A$ is left standardly stratified with respect to two posets $(Q_0,\unlhd)$ and $(Q_0,\unlhd')$ yielding the same set of standard modules so that $B$ is an exact Borel subalgebra of $(A, Q_0,\unlhd)$ but not of $(A,Q_0,\unlhd')$ (see \cite[Section A.2]{KKO14} for an example illustrating this). The same peculiar phemonemon may occur for exact Borel subalgebras of right standardly stratified algebras. However, if $B$ is an exact Borel subalgebra of a standardly stratified algebra $(A,Q_0,\unlhd)$ such that 
    \begin{enumerate}
    \setcounter{enumi}{4}
        \item\label{borel-cond5} $\Ext_B^1(L_\mathtt{i}^B,L_\mathtt{j}^B)\neq 0$ implies that $\Ext_B^1(A\otimes_B L_\mathtt{i}^B,A\otimes_B L_\mathtt{j}^B)\neq 0$ for every $\mathtt{i},\mathtt{j}\in Q_0$
    \end{enumerate}
    then none of the ambiguities mentioned before occur.

\begin{lem}
\label{lem:condition(5)}
Assume that $B$ is an exact Borel subalgebra of the left (resp. right) standardly stratified algebra $(A,Q_0,\unlhd)$ and that condition \eqref{borel-cond5} above holds. Then $B$ is an exact Borel subalgebra of the left (resp. right) standardly stratified algebra $(A,Q_0,\unlhd_{\ess})$.
\end{lem}

\begin{proof}
     Suppose that $B$ is an exact Borel subalgebra of the left standardly stratified algebra $(A,Q_0,\unlhd)$. Let $\unlhd_{\ess}$ denote the essential order associated to $(A,Q_0,\unlhd)$. It is enough to check that the essential order associated to $(B,Q_0,\unlhd)$ is coarser than $(Q_0,\unlhd_{\ess})$. For this, it is convenient to use an alternative description of the essential order due to Mazorchuk, see \cite[Lemma 43]{Kul23}, whose proof generalises straightforwardly to standardly stratified algebras. Let $\mathtt{i},\mathtt{j}\in Q_0$. The essential order for $A$ can be defined as the partial order induced by $\mathtt{i}\unlhd_{\ess} \mathtt{j}$ if $\Hom_A(\Delta_\mathtt{i},\Delta_\mathtt{j})\neq 0$ or $\Ext^1_A(\Delta_\mathtt{i},\Delta_\mathtt{j})\neq 0$. Since the standard modules for $B$ are the simple modules, only the second condition is relevant for $B$ and condition \eqref{borel-cond5} precisely states that this condition is the same for $A$ and $B$. 
\end{proof}

Exact Borel subalgebras satisfy a few fundamental properties. The next result is adapted from \cite[Theorem A]{Koe95} and it is an instance of the following general phenomenon: if an exact left adjoint functor between ``stratified categories'' preserves `standard-like' objects then its right adjoint preserves `costandard-like' objects (and vice-versa). We have tried to emphasise the generality of the argument in the proof.

\begin{prop}
\label{prop:restriction}
    Let $B$ be an exact Borel subalgebra of a left (resp.~right) standardly stratified algebra $(A,Q_0,\unlhd)$. Then $\Rest(\bar{\nabla}_\mathtt{i})=\bar{\nabla}_\mathtt{i}^B=I_\mathtt{i}^B$ for every $\mathtt{i}\in Q_0$ (resp.~$\Rest(\nabla_\mathtt{i})=\nabla_\mathtt{i}^B=I_\mathtt{i}^B$), where $\Rest\colon \Modu{A}\rightarrow\Modu{B}$ denotes the restriction functor.
\end{prop}

\begin{proof}
    We prove the result for left standardly stratified algebras. Fix $\mathtt{i} \in Q_0$. Consider $\mathtt{j}\in Q_0$ and a short exact sequence of $B$-modules
\[\begin{tikzcd}
0 \ar[r] & U \ar[r] & P \ar[r] & \Delta_\mathtt{j}^B \ar[r] & 0
\end{tikzcd} \]
with $P$ projective. Since $A\otimes_B-$ is exact and left adjoint to $\Rest$, we get the following diagram
\[\begin{tikzcd}
\Hom_{B}(P,\Rest(\bar{\nabla}_\mathtt{i}))\ar[d, "\sim", sloped] \ar[r] & \Hom_{B}(U,\Rest(\bar{\nabla}_\mathtt{i})) \ar[d, "\sim", sloped]\ar[r] & \Ext_{B}^{1}(\Delta_\mathtt{j}^B,\Rest(\bar{\nabla}_\mathtt{i}))  \ar[d, dashed,"\exists!"]\ar[r] & 0 \\
\Hom_{A}(A\otimes_B P,\bar{\nabla}_\mathtt{i}) \ar[r] & \Hom_{A}(A\otimes_B U,\bar{\nabla}_\mathtt{i}) \ar[r] & \Ext_{A}^{1}(A\otimes_B\Delta_\mathtt{j}^B,\bar{\nabla}_\mathtt{i}).&
\end{tikzcd} \]
As $A\otimes_B \Delta_\mathtt{j}^B\cong \Delta_\mathtt{j}$ and $\Ext_{A}^{1}(\Delta_\mathtt{j},\bar{\nabla}_\mathtt{i})=0$, both horizontal maps on the left hand side are epimorphisms. As a consequence, $\Ext_{B}^{1}(\Delta_\mathtt{j}^B,\Rest(\bar{\nabla}_\mathtt{i}))=0$. Since $\mathtt{j}\in Q_0$ is arbitrary, \cite[Theorem 1.6 (iv)]{AHLU00} implies that $\Rest(\bar{\nabla}_\mathtt{i})$ must be filtered by proper costandard $B$-modules. Now
\[\Hom_{B}(\Delta_\mathtt{j}^B,\Rest(\bar{\nabla}_\mathtt{i}))\cong \Hom_{A}(\Delta_\mathtt{j},\bar{\nabla}_\mathtt{i})\] 
so $\bar{\nabla}_\mathtt{i}^B$ is the only proper costandard module appearing in a $\bar{\nabla}^B$-filtration of $\Rest(\bar{\nabla}_\mathtt{i})$, with multiplicity
\[
\frac{\Dim{\Hom_{B}(\Delta_\mathtt{i}^B,\Rest(\bar{\nabla}_\mathtt{i}))}}{\Dim{\End_{B}(L_\mathtt{i}^B)}}=\frac{\Dim{\Hom_{A}(\Delta_\mathtt{i},\bar{\nabla}_\mathtt{i})}}{\Dim{\End_{B}(L_\mathtt{i}^B)}}=\frac{\Dim{\Hom_{A}(\Delta_\mathtt{i},\bar{\nabla}_\mathtt{i})}}{\Dim{\End_{A}(L_\mathtt{i})}}=1,
\]
where the penultimate equality uses condition \eqref{borel-cond4}. This shows that $\Rest(\bar{\nabla}_\mathtt{i})=\bar{\nabla}_\mathtt{i}^B=I_\mathtt{i}^B$.
\end{proof}

\begin{ex}
Note that even for quasi-hereditary algebras, the proposition above does not yield an equivalence of categories between $\add \bar{\nabla}$ and $\add I^B$ in general. For example, consider the Ringel dual of the dual extension algebra of the path algebra of a linearly oriented $\mathbb{A}_3$-quiver (see \cite[Example 23]{Maz10}). It is given as the path algebra of the quiver 
\[
\begin{tikzcd}
    \overset{\mathtt{2}}{\circ}\arrow[bend left]{r}{\delta}&\overset{\mathtt{1}}{\circ}\arrow[bend left]{l}{\gamma}\arrow[bend left]{r}{\beta}&\overset{\mathtt{3}}{\circ}\arrow[bend left]{l}{\alpha}
\end{tikzcd}
\]
with relations $\gamma\delta$, $\beta\alpha$, $\beta\delta\gamma\alpha$. This is a quasi-hereditary algebra with respect to the order $\mathtt{1}<\mathtt{2}<\mathtt{3}$ with regular exact Borel subalgebra generated by $\gamma$, $\beta$, $\beta\delta$. One can check that $\Kopf(\nabla_\mathtt{3})=L_\mathtt{1}$ while $\Kopf(I_\mathtt{3}^B)=(L_\mathtt{1}^B)^{\oplus 2}$. Thus, $\dim(\Hom_A(\nabla_\mathtt{3},\nabla_\mathtt{1}))=1$, but $\dim(\Hom_B(I_\mathtt{3}^B,I_\mathtt{1}^B))=2$, so the restriction functor cannot induce an equivalence of categories. 
\end{ex}

The following proposition generalises the main result of \cite{Conde2021b} from quasi-hereditary algebras to standardly stratified algebras. By \cite{Kle84}, the condition is equivalent to the fact that the dual coring of $B\subseteq A$ is normal, i.e. that it has a group-like element. 

\begin{prop}
\label{prop:normality}
Let $B$ be an exact Borel subalgebra of a standardly stratified algebra $(A,Q_0,\unlhd)$. The inclusion $\iota\colon B\hookrightarrow A$ splits as a morphism of right $B$-modules and the splitting can be chosen so that its kernel is a right ideal in $A$.
\end{prop}

\begin{proof}
We use Proposition \ref{prop:restriction} and proceed exactly like in \cite[Theorem 3.2]{Conde2021}. For completeness, we prove the result for left standardly stratified algebras. Let $\Rest'\colon \Modu{A\Op} \rightarrow \Modu{B\Op}$ denote the restriction functor on the right. Regard the embedding  of $B$ into $A$ as a morphism $\iota\colon B \hookrightarrow \Rest '(A)$ of right $B$-modules. The opposite algebra $A\Op$ of $A$ is right standardly stratified with respect to the same indexing poset $(Q_0, \unlhd)$ and $\bar{\Delta}_\mathtt{i}^{A\Op}=D(\bar{\nabla}_\mathtt{i})$. Using the standard duality and Proposition \ref{prop:restriction}, we get
	\[
	\Rest'( \bar{\Delta}_\mathtt{i}^{A\Op}) \cong \Rest'( D( \bar{\nabla}_\mathtt{i}) ) \cong D( \Rest ( \bar{\nabla}_\mathtt{i}) ) \cong D( I_\mathtt{i}^B) \cong P_\mathtt{i}^{B\Op}.
	\]
	Suppose that $B$ decomposes as $\bigoplus_{\mathtt{i}\in Q_0}(P_\mathtt{i}^{B\Op})^{\oplus a_\mathtt{i}}$ as a right $B$-module, for certain positive integers $a_\mathtt{i}$. There must exist some isomorphism
	\[\varphi\colon B\longrightarrow \Rest'\left( \bigoplus_{\mathtt{i}\in Q_0}( \bar{\Delta}_\mathtt{i}^{A\Op})^{\oplus a_\mathtt{i}} \right)\]
	in $\Modu{B\Op}$. Let $\overline{\varphi}\colon A\rightarrow \bigoplus_{i\in Q_0}(\bar{\Delta}_\mathtt{i}^{A\Op})^{\oplus a_\mathtt{i}}$ be the morphism in $\Modu{A\Op}$ given by $\overline{\varphi}(a)=\varphi(1_B)a$ for $a\in A$. Note that
\[
\varphi(b)=\varphi(1_B)b=\varphi(1_B)\iota(b)=\overline{\varphi}(\iota(b))=( \Rest'(\overline{\varphi})\circ \iota ) (b).
\]
This means that the map $\iota$ is a split monic in $\Modu{B\Op}$. Since the splitting epimorphism $\Rest'(\overline{\varphi})$ is the restriction of a morphism of right $A$-modules, its kernel is a right ideal in $A$.
\end{proof}

In \cite{BKK20, KM23} stronger notions of exact Borel subalgebras have been considered.

\begin{defn}\label{def:homological-regular}
Let $B$ be a (unital) subalgebra of an algebra $A$ and let $Q_0$ be a set labelling the simples over $B$. The subalgebra $B$ of $A$ is \emphbf{homological} if, for every $X$ and $Y$ in $\Modu{B}$, the maps
\[\
\Ext_{B}^{n}(X,Y) \longrightarrow \Ext_{A}^{n}(A\otimes_B X,A\otimes_B Y),
\]
induced by the exact functor $A\otimes_B - $, are isomorphisms for $n\geq 2$ and epimorphisms for $n=1$. Similarly, the subalgebra $B$ is \emphbf{regular} if the maps
\[\
\Ext_{B}^{n}(L^B_\mathtt{i},L^B_\mathtt{j}) \longrightarrow \Ext_{A}^{n}(A\otimes_B L^B_\mathtt{i},A\otimes_B L^B_\mathtt{j})
\]
are isomorphisms for every $n\geq 1$ and every $\mathtt{i},\mathtt{j}\in Q_0$.
\end{defn}

By \cite[Remark 3.5]{BKK20}, every regular subalgebra is homological (see also \cite[Lemma 2.5]{Conde2021b}). Note that by restricting to regular subalgebras, one can talk unambiguously about exact Borel subalgebras of quasi-hereditary algebras, without having to specify the left or right standardly stratified structure. Moreover, it follows from Lemma \ref{lem:condition(5)} that the essential order of a regular exact Borel subalgebra must be coarser than the essential order of $(A,Q_0,\unlhd)$, so the dependence on the choice of a poset mentioned before does not arise. The results in \cite{BPS23,KKO14} assure the existence of regular exact Borel subalgebras for left standardly stratified algebras, up to equivalence (recall Definition \ref{defn:equivalent}), and in \cite{Got24} a similar result is proved for right standardly stratified algebras: given any left (resp.~right) standardly stratified algebra $(A,Q_0,\unlhd)$ there exists an equivalent left (resp.~right) standardly stratified algebra $(A',Q_0,\unlhd)$ such that $A'$ has a regular exact Borel subalgebra. To be precise, only the existence of a homological exact Borel subalgebra is claimed there. The regularisation algorithm, described in \cite{KR77} gives a possibility to pass from a homological exact Borel subalgebra (corresponding to a bocs with projective kernel) to a regular exact Borel subalgebra (corresponding to a regular bocs, i.e.~one without superfluous arrows). For quasi-hereditary algebras, the representative of the equivalence class is unique up to isomorphism as shown via different methods in \cite{Conde2021} and \cite[Section 3.4]{KM23}. The latter paper even proves the uniqueness of the embedding of the basic regular exact Borel subalgebra into this specific Morita representative. Generalising to standardly stratified algebras and strengthening the previous uniqueness theorems, in \cite[Theorem 7.2]{RR25}, Rodriguez Rasmussen proved uniqueness of basic regular exact Borel subalgebras up to conjugation.

\section{Exact Borel subalgebras and compatible idempotents}
\label{sec:exact-borels-compatible}

Suppose that $A$ is a left standardly stratified algebra with respect to an indexing poset $(Q_0, \unlhd)$. Let $B$ be an exact Borel subalgebra of $(A, Q_0, \unlhd)$ and denote by $\iota\colon B \hookrightarrow A$ the corresponding embedding. An idempotent in $B$ is obviously an idempotent in $A$, but the corresponding supports in $B$ and $A$ usually differ. To avoid ambiguity and indicate whether an idempotent $e\in B$ is considered as an element of $B$ or $A$, we write $e$ or $\iota (e)$, respectively. In this situation, there are well-defined algebra morphisms
\[\iota_|\colon eBe\longrightarrow \iota(e)A\iota(e), \quad \iota^|\colon B/BeB \longrightarrow A / A \iota(e) A,\]
satisfying $\iota_|(b)=\iota(b)$ and $\iota^|(b+BeB)=\iota(b)+A\iota(e)A$. The map $\iota_|$ is injective, but the same does not generally hold for $\iota^|$ when $e\in B$ is an arbitrary idempotent.

\begin{ex}
	\label{ex:noninjective}
	Consider the path algebra
	\[
	A=\Bbbk \left( 
	\begin{tikzcd}[ column sep = small]
	\overset{\mathtt{1}}{\circ} \ar[r] & \overset{\mathtt{2}}{\circ} & \overset{\mathtt{3}}{\circ} \ar[r]  \ar[l]  & \overset{\mathtt{4}}{\circ}
	\end{tikzcd}\right) .
	\]
	The algebra $A$ is quasi-hereditary with respect to the natural order on $\{\mathtt{1},\mathtt{2},\mathtt{3},\mathtt{4}\}$. The Morita equivalent version
	\[A'=\begin{pmatrix}
	\Bbbk & \Bbbk & 0 & 0 & 0 \\
	0 & \Bbbk & 0 & 0 & 0 \\
	0 & \Bbbk & \Bbbk & \Bbbk & \Bbbk \\
	0 & \Bbbk & \Bbbk & \Bbbk & \Bbbk \\
	0 & 0 & 0 & 0 & \Bbbk
	\end{pmatrix}\Op\cong \End_{A}(P_\mathtt{1}\oplus P_\mathtt{2}\oplus P_\mathtt{3}^{\oplus 2}\oplus P_\mathtt{4})\Op
	\]
	of $A$ is still quasi-hereditary with respect to $(\{\mathtt{1},\mathtt{2},\mathtt{3},\mathtt{4}\},\leq)$ in a canonical way. According to \cite[Example 4.22]{Conde2021b}, the algebra $A'$ has an exact Borel subalgebra
	\[
	B= \left\{\begin{pmatrix}a&b&0&0&0\\0&c&0&0&0\\0&0&e&0&g\\0&b&i&a&k\\0&0&0&0&\ell\end{pmatrix}\,\bigg|\, a,b,c,e,g,i,k,\ell\in \Bbbk\right\}^{\op}\cong\Bbbk\left(\begin{tikzcd}[ column sep = small]
	\overset{\mathtt{1}}{\circ} \ar[r]\ar[rr, bend right] & \overset{\mathtt{2}}{\circ} & \overset{\mathtt{3}}{\circ}   \ar[r]  & \overset{\mathtt{4}}{\circ}
	\end{tikzcd}\right).
	\]
	   Here the trivial path $e_\mathtt{1}$ in the path algebra isomorphic to $B$ is mapped to the matrix $E_{11}+E_{44}\in B\subseteq A'$ (here $E_{ij}$ denotes the matrix whose only non-zero entry is a $1$ lying in position $(i,j)$). By slight abuse of notation we write $e_\mathtt{1}$ even for the element $E_{11}+E_{44}$ in $B$. With this convention, we see already that $\iota(e_\mathtt{1})\in A'$ has support $\{\mathtt{1},\mathtt{3}\}$, whereas $e_\mathtt{1}\in B$ has support $\{\mathtt{1}\}$. It is easy to check that $\dim (B/Be_\mathtt{1}B)=4$ and that $\dim(A'\iota(e_\mathtt{1})+A'\iota(e_\mathtt{1})E_{34})=10$ (note that the multiplication in $A'$ is opposite to the usual matrix multiplication). Since $A'\iota(e_\mathtt{1})+A'\iota(e_\mathtt{1})E_{34}$ is a subspace of $ A'\iota(e_\mathtt{1})A'$ (in fact both spaces coincide), then \[
	\dim(A'/A'\iota(e_\mathtt{1})A')\leq \dim(A')-\dim(A'\iota(e_\mathtt{1})+A'\iota(e_\mathtt{1})E_{34})=12-10=2.
	\]
	Given that $\dim (B/Be_\mathtt{1}B)> \dim(A'/A'\iota(e_\mathtt{1}A')$, the algebra morphism $\iota^|\colon B/Be_\mathtt{1}B \to A'/A'\iota(e_\mathtt{1})A'$ induced by $\iota\colon B\hookrightarrow A'$ cannot be injective. 
\end{ex}

\begin{lem}
	\label{lem:second}
	Let $(B, Q_0, \unlhd)$ be a right standardly stratified algebra with simple proper standard modules. If $e\in B$ is an idempotent supported in $Q_0'$ and $D(B/BeB)$ is $\nabla$-filtered, then $[Be:L_\mathtt{j}^B]\neq 0$ if and only if $\mathtt{j}\in Q_0'$.
\end{lem}
\begin{proof}
	The backwards implication is trivial. Suppose $[Be:L_\mathtt{j}^B]\neq 0$, that is $[Be:\bar{\Delta}_\mathtt{j}^B]\neq 0$ for $\mathtt{j}\in Q_0$. Lemma \ref{lem:comp-third'} (or rather the corresponding version for right standardly stratified algebras) and BGG reciprocity imply that $\mathtt{j}\in Q_0'$.
\end{proof}

As a side note, observe that, under the assumptions of the previous lemma, the idempotent $e\in B$ is actually compatible with the standardly stratified structure of $(B, Q_0, \unlhd)$. The same holds for the upcoming results.

\begin{lem}
	\label{lem:first}
	Let $B$ be an exact Borel subalgebra of a standardly stratified algebra $(A, Q_0, \unlhd)$. If $e\in B$ is an idempotent supported in $Q_0'$ and $D(B/BeB)$ is filtered by costandard $B$-modules, then $\iota(e)$ is also supported in $Q_0'$.
\end{lem}
\begin{proof}
	We prove the result for exact Borel subalgebras of left standardly stratified algebras. The argument works the same for exact Borel subalgebras of right standardly stratifed algebras by replacing proper costandard modules by costandard modules. Let $\mathtt{i} \in Q_0$. Note that 
	\[
	\iota (e) L_\mathtt{i} \cong \Hom_{A}(A\iota(e),L_\mathtt{i})\cong \Hom_{A}(A\otimes_B B e,L_\mathtt{i})\cong \Hom_{B}(B e,\Rest(L_\mathtt{i})).
	\]
	Recall that $\Rest(\bar{\nabla}_\mathtt{i})=\bar{\nabla}_\mathtt{i}^B=I_\mathtt{i}^B$ (see Proposition \ref{prop:restriction}), so $L_\mathtt{i}^B\subseteq \Rest(L_\mathtt{i})\subseteq \Rest(\bar{\nabla}_\mathtt{i})$. As a consequence, \[[L_\mathtt{i}^B:L_\mathtt{j}^B]\leq [\Rest(L_\mathtt{i}):L_\mathtt{j}^B]\leq[\Rest(\bar{\nabla}_\mathtt{i}):L_\mathtt{j}^B]=[\bar{\nabla}_\mathtt{i}^B:L_\mathtt{j}^B]\]
	for every $\mathtt{j}\in Q_0$. In particular, $[\Rest(L_\mathtt{i}):L_\mathtt{i}^B]\neq 0$. Moreover, $[\Rest(L_\mathtt{i}):L_\mathtt{j}^B]\neq 0$ implies that $[\bar{\nabla}_\mathtt{i}^B:L_\mathtt{j}^B]\neq 0$. As $e$ is supported in $Q_0'$, then
	\[\iota(e)L_\mathtt{i}\neq 0 \Leftrightarrow \Hom_B(B e,\Rest(L_\mathtt{i}))\neq 0 \Leftrightarrow \left(\exists \mathtt{j}\in Q_0'\colon[\Rest(L_\mathtt{i}):L_\mathtt{j}^B]\neq 0 \right).
	\]
	We claim that the equivalent statements above hold if and only if $\mathtt{i}\in Q_0'$. Since $[\Rest(L_\mathtt{i}):L_\mathtt{i}^B]\neq 0$, the backwards  implication holds as well. Conversely, if $[\Rest(L_\mathtt{i}):L_\mathtt{j}^B]\neq 0$ for some $\mathtt{j}\in Q_0'$, then $[\bar{\nabla}_\mathtt{i}^B:L_\mathtt{j}^B]\neq 0$, so $\mathtt{i}\in Q_0'$ by Lemma \ref{lem:comp-third'}. Therefore $\iota(e)$ has support $Q_0'$.
\end{proof}

Given an idempotent $e$ in an exact Borel subalgebra $B$ of a standardly startified algebra $(A, Q_0, \unlhd)$, one may construct the following diagram with recollements on the rows:
\begin{equation}
\label{eq:dotted-diagram}
\begin{tikzcd}[row sep=6.5 em, column sep=6.0 em]
\Modu{B/BeB}  \arrow[r, Green,line width=0.6 pt,no head, dotted, shift left] \arrow{r}[description]{i} & \Modu{B} \arrow{r}[description]{e(-)} \arrow[l, "\Homo{B}{B/BeB}{-}",bend left]   \arrow[d, Red,line width=0.6 pt,no head, dotted, shift left] \arrow[d,Green,line width=0.6 pt,no head, dotted, shift right] \arrow{d}[description]{A\otimes_B-} \arrow[l,"(B/BeB)\otimes_B-"', bend right, shift right] & \Modu{eBe} \arrow[l,"\Homo{eBe}{eB}{-}", bend left] \arrow[l ,"Be\otimes_{eBe}-"',bend right] \arrow[l ,Red, line width=0.6 pt,no head, dotted, bend right, shift right]                 \\
\Modu{A/A\iota(e)A} \arrow{r}[description]{i}                         & {\Modu{A}}  \arrow[l, "(A/A\iota(e)A)\otimes_{A}-"',bend right] \arrow[l,"\Homo{A}{A/A\iota(e)A}{-}", bend left] \arrow[r, Red,line width=0.6 pt,no head, dotted, shift left] \arrow[l,Green,line width=0.6 pt, no head, dotted, bend right, shift right]      \arrow{r}[description]{\iota(e)(-)}                                                         & \Modu{\iota(e)A\iota(e)} \arrow[l,"A\iota(e)\otimes_{\iota(e)A\iota(e)}-"' ,bend right, shift right] \arrow[l, "\Homo{\iota(e)A\iota(e)}{\iota(e)A}{-}", bend left]
\end{tikzcd}.
\end{equation}
We shall prove that the algebras $\iota(e)A\iota(e)$ and $A/A\iota(e)A$ still have exact Borel subalgebras, provided that the idempotent $e$ is adequately chosen.
For this, we will use the colocalisation subdiagrams of the recollements above. Similar results about $\Delta$-subalgebras of standardly stratified algebras can be proved using the localisation subdiagrams. We refer to \cite[Section 4]{Koe95} for the definition of $\Delta$-subalgebra of a quasi-hereditary algebra -- this is a dual counterpart of the notion of exact Borel subalgebra.

\begin{prop}
\label{prop:compatibility-subalgebra}
	Let $B$ be an exact Borel subalgebra of a left (resp.~right) standardly stratified algebra $(A, Q_0, \unlhd)$. Let $e\in B$ be an idempotent supported in $Q_0'$ and assume that the $B$-module $D(B/BeB)$ and the $A$-module $D(A/A\iota(e)A)$ are filtered by proper costandard modules (resp.~by costandard modules). Then $eBe$ and $\iota(e)A\iota(e)$ are both left (resp.~right) standardly stratified with respect to the subposet $(Q_0', \unlhd)$ and the algebra monomorphism
	\[\iota_|\colon eBe\lhook\joinrel\longrightarrow \iota(e)A\iota(e)\]
	turns $eBe$ into an exact Borel subalgebra of $(\iota(e) A\iota(e), Q_0', \unlhd)$. If $B$ is a homological exact Borel subalgebra of $(A, Q_0, \unlhd)$, then the exact Borel subalgebra $eBe$ of $(\iota(e)A\iota(e), Q_0', \unlhd)$ is also homological, and the corresponding result also holds for regular exact Borel subalgebras.
\end{prop}
\begin{proof}
	We prove the result for left standardly stratified algebras. For the right standardly stratified case the proof is analogous. Suppose that the idempotent $e\in B$ is such that the $B$-module $D(B/BeB)$ is filtered by proper costandard $B$-modules and the $A$-module $D(A/A\iota(e)A)$ is filtered by proper costandard $A$-modules. By Lemma \ref{lem:comp-fourth'}, $eBe$ is left standardly stratified with respect to $(Q_0',\unlhd)$. Using Lemmas \ref{lem:comp-fourth'} and \ref{lem:first}, we conclude in a similar manner that $\iota(e)A\iota(e)$ is left standardly stratified with respect to $(Q_0', \unlhd)$  
	
	By Lemma \ref{lem:comp-fourth'}, the standard modules over $eBe$ are given by $e \Delta_\mathtt{i}^B=e L_\mathtt{i}^B=L_\mathtt{i}^{eBe}$ for $\mathtt{i}\in Q_0'$, so $eBe$ is still a left standardly stratified algebra whose standard modules are simple. Moreover, the algebra monomorphism $\iota_|$ embeds $eBe$ into $\iota(e)A\iota(e)$, as previously remarked. The morphism $\iota_|$ turns every $\iota(e) A\iota(e)$-module into an $eBe$-module by restriction of the action. We shall prove that the corresponding induction functor $\iota(e)A\iota(e)\otimes_{eBe} -$ is exact and takes simples to the corresponding standards, thus establishing that $eBe$ is an exact Borel subalgebra of $(\iota(e)A\iota(e),Q_0', \unlhd)$. For this, consider the following composition of functors (follow the red dots in \eqref{eq:dotted-diagram})
	\[	F\colon 
	\begin{tikzcd}[ column sep = huge]
	\Modu{eBe} \ar[r, "Be\otimes_{eBe}-"] & \Modu{B} \ar[r, "A\otimes_B-"] & \Modu{A}\ar[r, "\iota(e)(-)"] & \Modu{\iota(e)A\iota(e)}
	\end{tikzcd}.\]
	We claim that $F$ is exact. Since $A\otimes_B-$ and $\iota(e)(-)$ are clearly exact, it suffices to show that $Be\otimes_{eBe}-$ is exact. Since $\Modu{B}$ has enough injectives, $Be\otimes_{eBe}-$ is exact if and only if $e(-)$ preserves injectives (see \cite[\href{https://stacks.math.columbia.edu/tag/015Z}{Tag 015Z}]{stacks-project}). If $\mathtt{j}\in Q_0'$, then $eI_\mathtt{j}^B=I_\mathtt{j}^{eBe}$. If $\mathtt{j}\in Q_0\setminus Q_0'$, then $eI_\mathtt{j}^B\cong\Hom_B(Be, I_\mathtt{j}^B)=0$ by Lemma \ref{lem:second}. Thus $F$ is exact. Note that $A\otimes_B Be \cong A \iota (e)$, so $F$ is naturally isomorphic to $\iota(e)A\iota(e)\otimes_{eBe}-$.
	
	We now prove that $F$ maps standards to standards. For $\mathtt{i} \in Q_0'$,  $Be\otimes_{eBe}-$ takes $\Delta_\mathtt{i}^{eBe}$ to $\Delta_\mathtt{i}^{B}$ (see Lemma \ref{lem:comp-fourth'}). Observe that $A\otimes_B \Delta_\mathtt{i}^{B}=\Delta_\mathtt{i}$ and $\iota(e)\Delta_\mathtt{i}=\Delta_\mathtt{i}^{\iota(e)A\iota(e)}$ by Lemma \ref{lem:comp-fourth'}. This means that $F$ takes the simple $eBe$-modules to the corresponding standard $\iota(e)A\iota(e)$-modules. This concludes the proof that $eBe$ is an exact Borel subalgebra of $\iota(e)A\iota(e)$.
	
	In order to show that homologicallity and regularity are preserved, it is enough to check that the functors $e(-)$ and $\iota(e)(-)$ give rise, respectively, to isomorphisms 
	\begin{gather*}
	\Ext_{B}^{n}(\Delta_\mathtt{i}^B,\Delta_\mathtt{j}^{B})\longrightarrow \Ext_{eBe}^{n}(e\Delta_\mathtt{i}^B,e\Delta_\mathtt{j}^B)=\Ext_{B}^{n}(\Delta_\mathtt{i}^{eBe},\Delta_\mathtt{j}^{eBe}),\\
	\Ext_{A}^{n}(\Delta_\mathtt{i},\Delta_\mathtt{j})\longrightarrow \Ext_{\iota(e)A\iota(e)}^{n}(\iota(e)\Delta_\mathtt{i},\iota(e)\Delta_\mathtt{j})=\Ext_{\iota(e)A\iota(e)}^{n}(\Delta_\mathtt{i}^{\iota(e)A\iota(e)},\Delta_\mathtt{j}^{\iota(e)A\iota(e)}),
	\end{gather*}
	for every $\mathtt{i},\mathtt{j}\in Q_0'$ and $n\geq 1$. The reason why both maps are isomorphisms is a consequence of Lemma \ref{lem:X-and-Y-infinity} and \cite[Theorem 3.2]{APT92} (cf.~Corollary \ref{cor:equivalence-eAe}  and its proof). 
\end{proof}

\begin{rmk}
\label{rmk:coideal}
    In the statement of Proposition \ref{prop:compatibility-subalgebra} and in the next results, we have tried to keep our base assumptions to a minimum. Observe that the hypotheses of Proposition \ref{prop:compatibility-subalgebra}, Lemma \ref{lem:injective} and Proposition \ref{prop:compatibility-quotient} are always satisfied whenever the support of the idempotent $e\in B$ is a coideal of $(Q_0,\unlhd)$ -- this follows from Theorem \ref{thm:defofcompatibility} and Lemma \ref{lem:first}.
\end{rmk}

As seen in Example \ref{ex:noninjective}, if $\iota\colon B\hookrightarrow A$ is the embedding of an exact Borel subalgebra $B$ into a standardly stratified algebra $(A, Q_0, \unlhd)$ and $e$ is an idempotent in $B$, then the induced algebra morphism $\iota^|\colon B/BeB\to A/A\iota(e)A$ is not necessarily injective. 
\begin{lem}
	\label{lem:injective}
	Let $B$ be an exact Borel subalgebra of a standardly stratified algebra $(A, Q_0, \unlhd)$. Let $e\in B$ be an idempotent such that the $B$-module $D(B/BeB)$ is filtered by costandard modules and $\iota(e)$ is compatible with the standardly stratified structure of $(A, Q_0, \unlhd)$. Then $A(BeB)=A\iota(e)A$ and the algebra morphism
	\[\iota^|\colon B/BeB\longrightarrow A/A\iota(e)A\]
	is injective.
\end{lem}

\begin{proof}
	We prove the result for left standardly stratified algebras. Denote the restriction functor on the right by $\Rest'\colon \Modu{A\Op} \to \Modu{B\Op}$ and write $\Rest\colon \Modu{A}\to\Modu{B}$ for the restriction functor on the left, as usual. Recall that $A\Op$ is a right standardly stratified algebra with respect to $(Q_0,\unlhd)$. Since $\Rest(\bar{\nabla}_\mathtt{i})=\bar{\nabla}_\mathtt{i}^B=I_\mathtt{i}^B$ (see Proposition \ref{prop:restriction}), then 
	\[\Rest'(\bar{\Delta}_\mathtt{i}^{A\Op})=\Rest'(D(\bar{\nabla}_\mathtt{i}))=D(\Rest(\bar{\nabla}_\mathtt{i}))=D(I_\mathtt{i}^B)=P_\mathtt{i}^{B\Op}.\]
	We get the following identities of $B\Op$-modules
	\begin{align*}
	A(BeB)&= \Rest'(A)(BeB)\cong \left(\bigoplus_{\mathtt{i}\in Q_0} ( P_\mathtt{i}^{B\Op})^{\oplus(A\Op:\bar{\Delta}_\mathtt{i}^{A\Op})}\right)eB\\
 &=\bigoplus_{\mathtt{i}\in Q_0} \left( e_\mathtt{i} BeB\right)^{\oplus(D(A):\bar{\nabla}_\mathtt{i})}=\bigoplus_{\mathtt{i}\in Q_0'} \left( e_\mathtt{i} BeB\right)^{\oplus(D(A):\bar{\nabla}_\mathtt{i})}=\bigoplus_{\mathtt{i}\in Q_0'}  (e_\mathtt{i} B)^{\oplus(D(A):\bar{\nabla}_\mathtt{i})},
	  \end{align*}
    where the penultimate equality follows from Lemma \ref{lem:second}.
    Recall that $\iota(e)$ is supported in $Q_0'$ by Lemma \ref{lem:first}. The $A$-submodule $D(A/A\iota(e)A)$ of $D(A)$ is therefore filtered by proper costandard modules of the form $\bar{\nabla}_\mathtt{i}$ with $\mathtt{i}\in Q_0 \setminus Q_0'$. Hence, the expression above reduces to 
    \[
        \bigoplus_{\mathtt{i}\in Q_0'}  (e_\mathtt{i} B)^{\oplus(D(A\iota(e)A):\bar{\nabla}_\mathtt{i})}.
    \]
	On the other hand,
 \[\Rest'(A\iota(e)A)=\bigoplus_{\mathtt{i}\in Q_0} \left( e_\mathtt{i} B\right)^{\oplus(D(A\iota(e)A):\bar{\nabla}_\mathtt{i})}=\bigoplus_{\mathtt{i}\in Q_0'} \left( e_\mathtt{i} B\right)^{\oplus(D(A\iota(e)A):\bar{\nabla}_\mathtt{i})}.\]
	Here, for the last equality, we use the following facts. First, note that $A\iota(e)A$ can be regarded as the largest $A\Op$-submodule of the $A\Op$-module $A$ generated by the projectives $P_\mathtt{i}^{A\Op}$ with $\mathtt{i}\in Q_0'$. Therefore, each proper standard module $\bar{\Delta}^{A\Op}_\mathtt{j}$ appearing in a $\bar{\Delta}^{A\Op}$-filtration of $A\iota(e)A$ must also appear in a $\bar{\Delta}^{A\Op}$-filtration of $P_\mathtt{i}^{A\Op}$ for some $\mathtt{i}\in Q_0'$. So, if $\bar{\nabla}_\mathtt{j}$ appears in a $\bar{\nabla}$-filtration of $D(A\iota(e)A)$, it must also appear in a $\bar{\nabla}$-filtration of $I_\mathtt{i}$ for some $\mathtt{i}\in Q_0'$ and Lemma \ref{lem:comp-third} assures that $\mathtt{j}\in Q_0'$. Since $A(BeB) \subseteq A\iota(e)A$, we must have an equality.
	
	Observe now that the algebra morphism $\iota^|$ defined previously can be obtained through the composition
	\[
	\begin{tikzcd}[column sep=large]
	B/BeB \ar[r, "\operatorname{can}_1"] & B\otimes_B (B/BeB) \ar[r, "\iota\otimes_B (B/BeB)"] & A\otimes_B (B/BeB)\ar[r, "\operatorname{can}_2"] & A/ A(BeB)=A/A\iota(e)A
	\end{tikzcd},\]
	where $\operatorname{can}_1$ and $\operatorname{can}_2$ are canonical isomorphisms. In order to prove that $\iota^|$ is injective it is enough to show that $\iota\otimes_B (B/BeB)$ is injective. In fact, the monic $\iota\colon B\hookrightarrow A$ splits as a morphism of right $B$-modules by Proposition \ref{prop:normality}, so its image under any functor starting in $\Modu{B\Op}$ must be still be a split monic. In particular $\iota\otimes_B (B/BeB)$ is certainly injective.
\end{proof}

\begin{rmk}
   Suppose the hypothesis of Lemma \ref{lem:injective} holds. It is natural to wonder how the bimodule $(BeB)A$ looks like and how it relates to $A\iota(e)A$ or $BeB$. It can happen that $(BeB)A$ is properly contained in $A\iota(e)A$. In fact, by considering Example \ref{ex:noninjective} and taking the idempotent $E_{55}\in B$ in there, one can check (again taking into account that multiplication is opposite to usual matrix multiplication) that $(BE_{55}B)A'$ is 3-dimensional whereas $A' \iota(E_{55}) A'$ is 6-dimensional (see \cite[Example 4.22]{Conde2021b} for more details on this example). By taking $e$ to be the identity on $B$ it is clear that, in general, $BeB$ is properly contained in $(BeB)A$ and in $A(BeB)$.
\end{rmk}

\begin{prop}
\label{prop:compatibility-quotient}
    Let $B$ be an exact Borel subalgebra of a left (resp.~right) standardly stratified algebra $(A, Q_0, \unlhd)$ and let $e\in B$ be an idempotent with support $Q_0'$. Assume that the $B$-module $D(B/BeB)$ is filtered by costandard modules and that $\iota(e)$ is compatible with the standardly stratified structure of $(A, Q_0, \unlhd)$. The algebras $B/BeB$ and $A/A\iota(e)A$ are both left (resp.~right) standardly stratified with respect to the subposet $(Q_0\setminus Q_0', \unlhd)$ and the algebra monomorphism
	\[\iota^|\colon B/BeB\lhook\joinrel\longrightarrow A/A\iota(e)A\]
	turns $B/BeB$ into an exact Borel subalgebra of $(A/A\iota(e)A, Q_0\setminus Q_0', \unlhd)$. If $B$ is a homological exact Borel subalgebra of $(A, Q_0, \unlhd)$, then the exact Borel subalgebra $B/BeB$ of $(A/A\iota(e)A, Q_0 \setminus Q_0', \unlhd)$ is also homological, and the corresponding result also holds for regular exact Borel subalgebras.
\end{prop}
\begin{proof}
	We prove the result for left standardly stratified algebras. By Lemma \ref{lem:comp-second}, $B/BeB$ is left standardly stratified with respect to $(Q_0 \setminus Q_0',\unlhd)$. Using Lemmas \ref{lem:comp-second} and \ref{lem:first}, we conclude in the same manner that $A/A\iota(e)A$ is left standardly stratified with respect to $(Q_0 \setminus Q_0',\unlhd)$. 
	
	By Lemma \ref{lem:comp-second}, the standard $B/BeB$-modules are given by $\Delta_\mathtt{i}^{B/BeB}=\Delta_\mathtt{i}^B=L_\mathtt{i}^B=L_\mathtt{i}^{B/BeB}$, $\mathtt{i}\in Q_0\setminus Q_0'$, so $B/BeB$ is still a left standardly stratified algebra with simple standard modules. Moreover, $\iota^|$ embeds $B/BeB$ into $A/A\iota(e)A$, as proved in Lemma \ref{lem:injective}.
	
	The monic $\iota^|$ turns every $A/A\iota(e)A$-module into a $B/BeB$-module by restriction of the action. We prove that the corresponding induction functor $(A/A\iota(e)A)\otimes_{B/BeB} -$ is exact and takes simples to the corresponding standards. For this, consider the following composition of functors (follow the green dots in \eqref{eq:dotted-diagram})
	\[	G\colon 
	\begin{tikzcd}[ column sep = huge]
	\Modu{B/BeB} \ar[r, "i"] & \Modu{B} \ar[r, "A\otimes_B-"] & \Modu{A}\ar[r, "A/A\iota(e)A\otimes_A-"] & \Modu{A/A\iota(e)A}
	\end{tikzcd}.\]
	We claim that $G$ is exact. Since the functors $i$ and $A\otimes_B-$ are exact, we just need to show that $A/A\iota(e)A\otimes_A-$ is exact on sequences in $\Modu{A}$ with modules in $\Ima{(A\otimes_B i(-))}$. It suffices to prove that all $A$-modules in $\Ima{(A\otimes_B i(-))}$ are actually $A/A\iota(e)A$-modules, i.e.~it is enough to show that $\iota(e) (A\otimes_B i(X))=0$ for every $X$ in $\Modu{B/BeB}$. Note that
	\[\iota(e) (A\otimes_B i(B/BeB))\cong\iota(e)(A/A(BeB))=\iota(e)(A/A\iota(e)A)=0,
	\]
	where the second equality follows from Lemma \ref{lem:injective}. Since every module in $\Modu{B/BeB}$ is generated by $B/BeB$ and the composition $A\otimes_B i(-)$ yields an exact functor which preserves all coproducts, we conclude that $\iota(e) (A\otimes_B i(X))=0$ for every $X$ in $\Modu{B/BeB}$. Hence $G$ is exact.
	
	We check that $G$ maps standards to standards. Fix $\mathtt{i}\in Q_0 \setminus Q_0'$. Note that $i(\Delta_\mathtt{i}^{B/BeB})=\Delta_\mathtt{i}^{B}$ and observe that $A\otimes_B \Delta_\mathtt{i}^{B}=\Delta_\mathtt{i}$ as $B$ is an exact Borel subalgebra. Finally, $A/A\iota(e)A\otimes_A \Delta_\mathtt{i} =\Delta_\mathtt{i}^{A/A\iota(e)A}$ by Lemma \ref{lem:comp-first}. Now, observe that
    \begin{align*}
    G&=A/A\iota(e)A \otimes_A A \otimes_B i(-) \cong A/A\iota(e)A \otimes_B i(-)\cong A\otimes_B B/BeB \otimes_B i(-) \\
    &\cong A\otimes_B B/BeB \otimes_{B/BeB}-\cong A/A\iota(e)A \otimes_{B/BeB}-.
  \end{align*}
    Hence, $B/BeB$ is an exact Borel subalgebra of $(A/A\iota(e)A,Q_0 \setminus Q_0', \unlhd)$.
	
	As $e$ is compatible with $(B,Q_0,\unlhd)$ and $\iota(e)$ is compatible with $(A,Q_0,\unlhd)$, both inclusion functors in \eqref{eq:dotted-diagram} are $\infty$-homological embeddings (see the end of Section \ref{sec:compatible-idempotents}) and they also preserve standard objects (see Lemma \ref{lem:comp-second}). In particular, for every $\mathtt{i}, \mathtt{j}\in Q_0\setminus Q_0'$ and $n\geq 1$, they give rise to isomorphisms
	\begin{gather*}
	\Ext_{B/BeB}^{n}(L_\mathtt{i}^{B/BeB},L_\mathtt{j}^{B/BeB})\longrightarrow \Ext_{B}^{n}(i(L_\mathtt{i}^{B/BeB}),i(L_\mathtt{j}^{B/BeB}))=\Ext_{B}^{n}(L_\mathtt{i}^{B},L_\mathtt{j}^{B}),\\
	\Ext_{A/A\iota(e)A}^{n}(\Delta_\mathtt{i}^{A/A\iota(e)A},\Delta_\mathtt{j}^{A/A\iota(e)A}) \longrightarrow \Ext_{A}^{n}(i(\Delta_\mathtt{i}^{A/A\iota(e)A}),i(\Delta_\mathtt{j}^{A/A\iota(e)A}))=\Ext_{A}^{n}(\Delta_\mathtt{i},\Delta_\mathtt{j}).
	\end{gather*}
	As a result, homologicallity and regularity are preserved.
\end{proof}

Up to this point, in this section we have started with an idempotent in the exact Borel subalgebra. In the following theorem, we show that when starting instead with an idempotent in the quasi-hereditary algebra, one can find a corresponding one in the exact Borel subalgebra.

\begin{thm}
\label{thm:last_firstpart}
    Let $B$ be an exact Borel subalgebra of a left (resp. right) standardly stratified algebra $(A, Q_0, \unlhd)$. Let $e\in A$ be an idempotent supported in some subset $Q_0'$ of $Q_0$ and assume that one of the following conditions hold:
    \begin{enumerate}
        \item\label{tag1} $Q_0'$ is a coideal of $(Q_0, \unlhd)$;
        \item\label{tag2} $B$ is such that $\Ext_B(L_\mathtt{i}^B,L_\mathtt{j}^B)\neq 0$ implies that $\Ext_B(A\otimes_B L_\mathtt{i}^B,A\otimes_B L_\mathtt{j}^B)\neq 0$ for every $\mathtt{i},\mathtt{j}\in Q_0$ and $e$ is compatible with the standardly stratified structure of $(A,Q_0,\unlhd)$.
   \end{enumerate}
    Then, any idempotent $e'\in B$ supported in $Q_0'$ is compatible with the standardly stratified structure of $(B,Q_0,\unlhd)$. In this case, the embedding of $B$ into $A$ induces an algebra monomorphism
    \[\iota_|:e'Be'\lhook\joinrel\longrightarrow  \iota(e')A\iota(e'){\sim} eAe\]
    that turns $e'Be'$ into an exact Borel subalgebra of the left (respectively right) standardly stratified algebra $(\iota(e')A\iota(e'), Q_0\setminus Q_0', \unlhd)$, which is equivalent to $(eAe,Q_0\setminus Q_0', \unlhd)$ in the sense of Definition \ref{defn:equivalent}, and also an algebra monomorphism 
    \[\iota^|:B/Be'B\lhook\joinrel\longrightarrow  A/A\iota(e')A=A/ AeA\]
    that turns $B/Be'B$ into an exact Borel subalgebra of the left (resp.~right) standardly stratified algebra $(A/A e A,  Q_0', \unlhd)$.  Moreover, $A(Be'B)=AeA=A\iota(e')A$. 
    Furthermore, if $B$ is homological then so are $e'Be'$ and $B/Be'B$, and the corresponding compatibility result also holds for $B$ regular.
\end{thm}

\begin{proof}
Suppose first that condition \eqref{tag1} in the statement of the theorem is satisfied.  As explained in Remark \ref{rmk:coideal}, it follows from Theorem \ref{thm:defofcompatibility} and Lemma \ref{lem:first} that the support of $\iota(e')\in A$ is $Q_0'$ for any idempotent $e'\in B$ supported in the coideal $Q_0'$. Furthermore, Proposition \ref{prop:compatibility-subalgebra}, Lemma \ref{lem:injective} and Proposition \ref{prop:compatibility-quotient} hold for the idempontent $e'\in B$. Hence, the result follows in this case. Suppose now that condition \eqref{tag2} holds. By Lemma \ref{lem:condition(5)}, the essential order associated to $(B,Q_0,\unlhd)$ is coarser than the essential order associated to $(A,Q_0,\unlhd)$. As $e\in A$ is compatible with the standardly stratified structure of $(A,Q_0,\unlhd)$, $Q_0'$ must be a coideal of the essential order associated to $(A,Q_0,\unlhd)$. So $Q_0'$ is also a coideal of the essential order associated to $(B,Q_0,\unlhd)$ and therefore $e'\in B$ is compatible with the standardly stratified structure of $(B,Q_0,\unlhd)$. As a consequence of Theorem \ref{thm:defofcompatibility} and Lemma \ref{lem:first}, the previous results (namely, Proposition \ref{prop:compatibility-subalgebra}, Lemma \ref{lem:injective} and Proposition \ref{prop:compatibility-quotient}) may be applied in this situation.
\end{proof}

One may wonder whether the following converse of Theorem \ref{thm:last_firstpart} holds. If $e$ is an idempotent compatible with the standardly stratified structure of an algebra $(A,Q_0, \unlhd)$ and both $A/AeA$ and $eAe$ have an exact Borel subalgebra, does this imply that $A$ also has an exact Borel subalgebra? Such tempting conjecture fails dramatically. The next example provides a negative answer to the question.

\begin{ex}
    Consider the basic algebra $A$ in Example \ref{ex:contraexample_steffen} and the idempotent $e=e_{\mathtt 3}+e_{\mathtt 4}\in A$. Note that $e$ is supported in a coideal of $(\{\mathtt{1},\mathtt{2},\mathtt{3},\mathtt{4}\}, \leq)$, so it is compatible with the quasi-hereditary structure of $A$. Both quasi-hereditary algebras $A/AeA$ and $eAe$ contain (regular) exact Borel subalgebras, since both only have two isomorphism classes of simple modules (see \cite[Theorem 4.58]{Kul17} and also \cite[Example 4.20]{Conde2021b}). However, $(A,\{\mathtt{1},\mathtt{2},\mathtt{3},\mathtt{4}\}, \leq)$ does not contain any exact Borel subalgebra (see \cite[Example 2.3]{Koe95}).
\end{ex}

An equivalent version of that same example also provides an answer to the question whether the Morita equivalence in Theorem \ref{thm:last_firstpart} can be chosen to be an isomorphism:

\begin{ex}
  Let $A$ be as before and consider the equivalent algebra $A'=\End_{A}(A\oplus P_\mathtt{3})^{\op}$. According to \cite[Appendix A.3]{KKO14}, this algebra has an exact Borel subalgebra $B$ in contrast to its basic representative. Thus taking an idempotent $e\in A'$ such that $eA'e\cong A$, there won't exist $e'\in B$ such that $e'Be'$ is an exact Borel subalgebra of $eA'e\cong A$, since $A$ does not have an exact Borel subalgebra. 
\end{ex}

\section{Decomposition multiplicities}
\label{sec:decomposition-multiplicities}

Suppose within this subsection that the underlying field $\Bbbk$ is algebraically closed. Let $[(A,Q_0,\unlhd)]$ be the class of all quasi-hereditary algebras equivalent to a quasi-hereditary algebra $(A,Q_0,\unlhd)$ (recall Defintion \ref{defn:equivalent}).

\begin{rmk}
Note that the equivalence class $[(A,Q_0,\unlhd)]$ does not depend on whether we regard $A$ as a left standardly stratified algebra or as a right standardly stratified algebra. This is because for quasi-hereditary algebras, the standard and the proper standard modules coincide. Alternatively, one could regard both left and right standardly stratified algebras with the corresponding classes of standard, respectively proper standard modules, at the same time (e.g. in the context of mixed standardly stratified algebras mentioned before) and adapt Definition \ref{defn:equivalent} to this context. If $\Delta_\mathtt{i}\neq \bar{\Delta}_\mathtt{i}$, then the two can be distinguished since only the latter has $\Bbbk$ as endomorphism ring.  
\end{rmk}

A necessary and sufficient criterion for a quasi-hereditary algebra to have a regular exact Borel subalgebra is given in \cite[Theorem 4.18]{Conde2021b}. Namely, a quasi-hereditary algebra $(A, Q_0, \unlhd)$ contains a regular exact Borel subalgebra if and only if the linear system $V_{[(A, Q_0, \unlhd)]}x=(\dim L^A_\mathtt{i})_{\mathtt{i}\in Q_0}$ has a solution consisting of positive integers. The matrix $V_{[(A, Q_0, \unlhd)]}$ in the linear system is invertible with non-negative integral entries and, up to simultaneous permutation of rows and columns, it is lower triangular with ones on the diagonal and zeros on the lower diagonal. Note that $V_{[(A, Q_0, \unlhd)]}$ can be used to determine all the quasi-hereditary algebras equivalent to $(A, Q_0, \unlhd)$ which have a regular exact Borel subalgebra and also to compute the Cartan matrix of regular exact Borel subalgebras of quasi-hereditary algebras equivalent to $(A, Q_0, \unlhd)$ (see \cite[Theorems 4.16, 4.23]{Conde2021b}). 

The matrix $V_{[(A, Q_0, \unlhd)]}$ can be computed by an inductive procedure. The rows and columns of $V_{[(A, Q_0, \unlhd)]}$ are parametrised by $Q_0$, i.e.~$V_{[(A, Q_0, \unlhd)]}=(v_{\mathtt{i} \mathtt{j}})_{\mathtt{i},\mathtt{j}\in Q_0}$ where $v_{\mathtt{i} \mathtt{j}}$ is the $\mathtt{j}$th coordinate of an element $v_{\mathtt{i}}$ of free $\Z$-module ${\Z}^{Q_0}$ on $Q_0$. The sequence of elements $(v_{\mathtt{i}})_{\mathtt{i}\in Q_0}$ is defined recursively through the identity
\begin{equation}
\label{eq:recursivevector}
v_{\mathtt{i}}={\epsilon}_{\mathtt{i}}+\sum_{\substack{\mathtt{j},\mathtt{k}\in Q_0\\ \mathtt{k}\unlhd \mathtt{j} \lhd \mathtt{i}}}[\nabla_{\mathtt{j}}:L_{\mathtt{k}}] \dim( \Hom_{A}(\Delta_{\mathtt{j}},\Delta_{\mathtt{i}}) )v_{\mathtt{k}}  - \sum_{\substack{\mathtt{j}\in Q_0 \\ \mathtt{j} \lhd \mathtt{i}}}[\Delta_{\mathtt{i}}:L_{\mathtt{j}}]v_{\mathtt{j}} 
\end{equation}
where $\{{\epsilon}_{\mathtt{i}} \mid \mathtt{i}\in Q_0\}$ constitutes the standard basis of ${\Z}^{Q_0}$. Since the input data for this algorithm only depends on the equivalence class $[(A, Q_0, \unlhd)]$ of $(A, Q_0, \unlhd)$, then so does the matrix $V_{[(A, Q_0, \unlhd)]}$. The entry $(\mathtt{i},\mathtt{j})$ of $V_{[(A, Q_0, \unlhd)]}$ turns out to coincide with $[\Rest(L_\mathtt{i}^{A'}):L_\mathtt{j}^B]$, where $(A',Q_0,\unlhd)$ is any quasi-hereditary algebra equivalent to $(A, Q_0,\unlhd)$ in the sense of Definition \ref{defn:equivalent}  that has a regular exact Borel subalgebra $B$ (\cite[Theorem 4.6]{Conde2021b}). Here, the labelling poset for $A'$ is chosen to coincide with that of $A$ in such a way that the equivalence $\mathcal{F}(\Delta^A) \to \mathcal{F}(\Delta^{A'})$ of exact categories sends $\Delta_{\mathtt{i}}^A$ to $\Delta_{\mathtt{i}}^{A'}$ for all $\mathtt{i}\in Q_0$.

\begin{prop}
\label{prop:compatibility-matrix}
    Let $(A,Q_0, \unlhd)$ be a quasi-hereditary algebra and suppose that the underlying field $\Bbbk$ is algebraically closed. Let $e\in A$ be an idempotent supported in $Q_0'$ which is compatible with the quasi-hereditary structure of $(A,Q_0, \unlhd)$ and consider associated the matrix $V_{[(A, Q_0, \unlhd)]}=(v_{\mathtt{i}\mathtt{j}})_{\mathtt{i},\mathtt{j}\in Q_0}$. Then
    \[V_{[(A, Q_0, \unlhd)]}=
\begin{pmatrix}
    V_{[(A/AeA, Q_0\setminus Q_0', \unlhd)]} & 0 \\
    * & V_{[(eAe, Q_0', \unlhd)]}
\end{pmatrix}.
\]
\end{prop}

\begin{rmk}
    Note that Proposition \ref{prop:compatibility-matrix} can be proved directly via the recursive formula \eqref{eq:recursivevector}. Using this method, it is not hard to check that the upper left corner of $V_{[(A, Q_0, \unlhd)]}$ indeed coincides with $V_{[(A/AeA, Q_0\setminus Q_0', \unlhd)]}$ (to verify this, recall Lemmas \ref{lem:comp-second} and \ref{lem:comp-second'}). Showing that the right lower corner of $V_{[(A, Q_0, \unlhd)]}$ is $V_{[(eAe, Q_0', \unlhd)]}$ follows essentially from Corollary \ref{cor:equivalence-eAe}. 
     Here we prove Proposition \ref{prop:compatibility-matrix} using a different strategy, namely by resorting to Theorem \ref{thm:last_firstpart} and to the fact that the entries of $V_{[(A, Q_0, \unlhd)]}$ record the composition multiplicities of restricted modules.
\end{rmk}

\begin{proof}[Proof of Proposition \ref{prop:compatibility-matrix}]
    Recall that the results in \cite{KKO14} guarantee the existence of regular exact Borel subalgebras for quasi-hereditary algebras, up to equivalence. Using  Lemma \ref{lem:comp-second}, Corollary \ref{cor:equivalence-eAe} and the fact that the matrices $V_{[(A,Q_0,\unlhd)]}$, $V_{[(A/AeA,Q_0\setminus Q_0',\unlhd)]}$ and $V_{[(eAe, Q_0',\unlhd)]}$ are invariants of the equivalence classes of $(A,Q_0,\unlhd)$, $(A/AeA,Q_0\setminus Q_0',\unlhd)$ and $(eAe, Q_0',\unlhd)$, respectively, we may suppose, without loss of generality, that $(A,Q_0, \unlhd)$ contains a regular exact Borel subalgebra $B$ and that $e\in A$  is compatible with the quasi-hereditary structure of $A$ and has support $Q_0'$. By Theorem \ref{thm:last_firstpart}, there exists an idempotent $e' \in B$, supported in $Q_0'$ and compatible with the quasi-hereditary structure of $(B,Q_0, \unlhd)$, so that $\iota(e')\in A$ is also supported in $Q_0'$ and $B/B e' B$ and $e' B e'$ are regular exact Borel subalgebras of the quasi-hereditary algebras $(A/A\iota(e')A, Q_0 \setminus Q_0', \unlhd)$ and $(\iota(e')A\iota(e'), Q_0', \unlhd)$, respectively. Note that the quasi-hereditary algebra $(eAe,Q_0',\unlhd)$ is equivalent to $(\iota(e')A\iota(e'), Q_0', \unlhd)$ and $A/A\iota(e')A=A/AeA$. Denote the restriction functors by $\Rest\colon \Modu{A} \to \Modu{B}$, $\Rest^|\colon \Modu{A/A \iota(e') A} \to \Modu{B/B e' B}$ and $\Rest_ |\colon \Modu{\iota(e')A\iota(e')} \to \Modu{ e' Be'}$, respectively. Let $e_{\mathtt j}\in B$ be some idempotent such that $Be_{\mathtt j} \cong P_{\mathtt j}^B$. Recall that 
    \[
    v_{\mathtt i \mathtt j}=[\Rest(L_{\mathtt i}):L_{\mathtt j}^B]=\dfrac{\dim(\Hom_B(B e_{\mathtt j}, \Rest (L_{\mathtt i})))}{\dim(\Hom_B(B e_{\mathtt j}, L_{\mathtt j}^B))}=\dfrac{\dim(\Hom_A(A \iota(e_{\mathtt j}), L_{\mathtt i}))}{\dim(\Hom_B(B e_{\mathtt j}, L_{\mathtt j}^B))}=\dfrac{\dim(\iota(e_{\mathtt j})L_{\mathtt i})}{\dim(e_{\mathtt j}L_{\mathtt j}^B)}.
    \]
    For $\mathtt i, \mathtt j \in Q_0 \setminus Q_0'$ we have that $\iota(e_{\mathtt j})L_{\mathtt i}=\iota^|(e_{\mathtt j}+Be' B)L_{\mathtt i}^{A/A\iota (e') A}$ and $e_{\mathtt j}L_{\mathtt j}^B=(e_{\mathtt j}+Be' B)L_{\mathtt j}^{B/Be' B}$, therefore $v_{\mathtt i \mathtt j}=[\Rest^|(L_{\mathtt i}^{A/A\iota(e')A}):L_{\mathtt j}^{B/Be' B}]$ and consequently $v_{\mathtt i \mathtt j}$ coincides with the entry $(\mathtt{i},\mathtt{j})$ of $V_{[(A/AeA, Q_0\setminus Q_0', \unlhd)]}$. For $\mathtt{j} \in Q_0'$ we may suppose without loss of generality (otherwise replace $e_\mathtt{j}$ by $e' e_\mathtt{j}e'$) that the idempotent $e_{\mathtt j}\in B$ satisfies $e' e_{\mathtt j}e'= e_{\mathtt j}$, so that $e_{\mathtt j}\in e' B e'$. Under this assumption, the identities $\iota(e_{\mathtt j})L_{\mathtt i}=\iota(e_{\mathtt j})\iota(e') L_{\mathtt i}=\iota_|(e_{\mathtt j})L_{\mathtt i}^{\iota(e') A\iota(e')}$ and $e_{\mathtt j}L_{\mathtt j}^B=e_{\mathtt j}e' L_{\mathtt j}^{B}=e_{\mathtt j}L_{\mathtt j}^{e' Be'}$ hold for $\mathtt i, \mathtt j \in Q_0'$, hence we have $v_{\mathtt i \mathtt j}=[\Rest_|(L_{\mathtt i}^{\iota(e')A\iota(e')}):L_{\mathtt j}^{e' B e'}]$ and so $v_{\mathtt i \mathtt j}$ coincides with the entry $(\mathtt{i},\mathtt{j})$ of $V_{[(eAe, Q_0\setminus Q_0', \unlhd)]}$.
\end{proof}

According to \cite[Theorem A (4)]{Conde2021} (see also \cite[Corollary 3.25]{KM23} and \cite{RR25}), for each equivalence class of quasi-hereditary algebras, there is a unique representative having a basic regular exact Borel subalgebra. 

\begin{prop}
\label{prop:last}
Let $(A,Q_0,\unlhd)$ be a quasi-hereditary algebra and suppose that the underlying field $\Bbbk$ is algebraically closed. Let the multiplicities $\ell_\mathtt{i}^A$ for $\mathtt{i}\in Q_0$ be defined in such a way that
\[R=\End_A\left(\bigoplus_{\mathtt{i}\in Q_0}P_\mathtt{i}^{\oplus\ell_{\mathtt{i}}^A}\right)^{\op}\]
is the unique representative of $[(A,Q_0,\unlhd)]$ having a basic regular exact Borel subalgebra. Let $e\in A$ be an idempotent supported in $Q_0'$, which is compatible with the quasi-hereditary structure of $(A,Q_0,\unlhd)$. Then, $\ell_\mathtt{i}^{eAe}=\ell_\mathtt{i}^A$ for $\mathtt{i}\in Q_0'$ and $\ell_\mathtt{i}^A\geq \ell_\mathtt{i}^{A/AeA}$ for $\mathtt{i}\in Q_0\setminus Q_0'$.
\end{prop}

\begin{proof}
According to \cite[Theorem B]{Conde2021}, the multiplicities $\ell_i^A$ are given by 
\[\ell_\mathtt{i}=\dim L_\mathtt{i}^R=\sum_{\mathtt{j}\in Q_0}[\Rest(L_{\mathtt{i}}^R):L_{\mathtt{j}}^B]=\sum_{\mathtt{j}\in Q_0}v_{\mathtt{i}\mathtt{j}}\]
where $B$ is the basic exact Borel subalgebra of $R$ and $V=(v_{\mathtt{i}\mathtt{j}})$ is the matrix in Proposition \ref{prop:compatibility-matrix}. Given the shape of $V$ therein, the claim follows. 
\end{proof}

We conclude our paper by giving a concrete instance to which the proposition is applicable, namely principal blocks of BGG category $\mathcal{O}$. For the definitions and the basic theory of BGG category $\mathcal{O}$ we refer the reader to \cite{Dix96, Hum08}.

\begin{cor}
\label{cor:last}
Let $\mathfrak{g}$ be a finite-dimensional semisimple complex Lie algebra. Let $A$ be the quasi-hereditary algebra corresponding to the principal block of BGG category $\mathcal{O}$ associated to $\mathfrak{g}$ which has a basic regular exact Borel subalgebra. 
\begin{enumerate}[(i)]
\item Then $A$ is basic if and only if $\mathfrak{g}=\mathfrak{sl}_2$.
\item Otherwise, there exists an indecomposable projective whose multiplicity in $A$ is at least $3$. If $\mathfrak{g}$ is simple, there even exists an indecomposable projective whose multiplicity is at least $9$. 
\end{enumerate}
\end{cor}

\begin{proof}
That the basic algebra associated to the principal block of $\mathfrak{sl}_2$ has a basic regular exact Borel subalgebra is well-known, see e.g.~\cite[Appendix A.1]{KKO14}. 
Assume that $\mathfrak{g}$ is not $\mathfrak{sl}_2$, i.e.~its rank is greater than $1$. In this case it is possible to choose a parabolic $\mathfrak{p}=\mathfrak{l}\oplus \mathfrak{u}$ with Levi subalgebra $\mathfrak{l}$ whose derived subalgebra $\mathfrak{s}$ has rank $2$. In case that $\mathfrak{g}$ is simple, $\mathfrak{s}$ can be chosen to be simple. In this situation, there is a Serre subquotient of the principal block of BGG category $\mathcal{O}(\mathfrak{g})$ which is equivalent to $\mathcal{O}(\mathfrak{l})\simeq \mathcal{O}(\mathfrak{s})$, whose labels of the corresponding simple modules form a coideal in the Weyl group, equipped with the Bruhat order, see \cite[Section 6.2]{CMZ19}. In the case of module categories of finite-dimensional algebras (to which blocks of BGG category $\mathcal{O}$ are equivalent to), Serre subcategories and Serre quotients correspond to idempotent quotients $A/AeA$ and idempotent subalgebras $eAe$, respectively. Therefore, the preceding proposition applies and it suffices to show the claim about the multiplicity of projective indecomposables in $A$ for the (semi)simple Lie algebras of rank $2$, i.e.~Dynkin types $\mathbb{A}_2$, $\mathbb{B}_2$, $\mathbb{G}_2$ (and $\mathbb{A}_1\times \mathbb{A}_1$). For the first three, in \cite[Example 1.56]{Kul23} the multiplicities of indecomposable projectives in $A$ have been computed. They are given by the formula 
\[\ell_w=\begin{cases}1&\text{if $w$ is minimal}\\3^{\operatorname{ht}(w)-1}&\text{else,}\end{cases}\]
where $\operatorname{ht}(w)$ denotes the height of $w$ in the Weyl group $W$ equipped with the Bruhat order. For Dynkin types $\mathbb{A}_2$, $\mathbb{B}_2$, and $\mathbb{G}_2$, the maximal height is $3$, $4$, and $6$, respectively. In particular, in each of the cases, there is an indecomposable projective with multiplicity at least $9$. In the remaining  case, $\mathbb{A}_1\times \mathbb{A}_1$, the principal block is isomorphic to the tensor product of the principal block of $\mathfrak{sl}_2$ with itself. For the quiver and relations for the corresponding algebra as well as the standard modules, the reader can e.g.~consult \cite[Section 3.1]{BKM01}. Using a similar calculation as in \cite[Example 1.56]{Kul23}, one can show that $\ell_w=3$ for $w$ maximal in the Bruhat order for $\mathbb{A}_1\times \mathbb{A}_1$, noting that all the relevant multiplicities $[\nabla_{\mathtt{j}}:L_{\mathtt{i}}]$, $\dim\Hom(\Delta_\mathtt{i},\Delta_\mathtt{j})$ and $[\Delta_{\mathtt{i}}:L_{\mathtt{j}}]$ are equal to one.
\end{proof}

Showing just (i) in the above corollary is in fact easier using the criterion in \cite[Proposition 5.1 (4)]{Conde2021b}. Namely, a basic quasi-hereditary algebra has an exact Borel subalgebra if and only if $\rad{\Delta_\mathtt{i}}\in \mathcal{F}(\nabla)$ for all $\mathtt{i}\in Q_0$. For a block of BGG category $\mathcal{O}$ associated to a finite-dimensional semisimple complex Lie algebra, the unique projective Verma module $\Delta_\mathtt{i}$ has simple socle given by the unique simple Verma module. Furthermore, this simple appears in the projective Verma module with multiplicity $1$.  In addition, for the principal block, the projective Verma module has length at least the order of the corresponding Weyl group $W$. Thus, if $\mathfrak{g}\neq \mathfrak{sl}_2$, and if the radical of the projective Verma module had a $\nabla$-filtration, besides the simple dual Verma module, it would have to include at least one additional dual Verma module. However, every dual Verma module has the simple Verma module as one of its composition factors, a contradiction.

\bibliographystyle{alpha}
\bibliography{publication}

\end{document}